\documentclass[11pt,a4paper]{article}
\usepackage{amsfonts}
\usepackage{latexsym}
\usepackage{cite}
\usepackage{amsmath,amsfonts,latexsym,amssymb}
\usepackage[mathscr]{eucal}
\usepackage{cases,color}
\usepackage{amsthm}

\usepackage[bf,small]{caption2}
\usepackage{float}
\usepackage{graphicx}
\usepackage{amsmath}
\usepackage{amssymb}
\usepackage[all]{xy}

\newtheorem{theorem}{theorem}[section]

\newtheorem{conv}[theorem]{Convention}

\newtheorem{defn}[theorem]{Definition}

\newtheorem{lem}[theorem]{Lemma}
\newtheorem{nota}[theorem]{Notation}

\newtheorem{prop}[theorem]{Proposition}

\newtheorem{rmk}[theorem]{Remark}
\newtheorem{thm}[theorem]{Theorem}

\begin{document}

\title{\vspace{-2cm}\textbf{Presentations of skein algebras of genus 2}}
\author{\Large Haimiao Chen}
\date{}
\maketitle

\begin{abstract}
  We give a presentation for the Kauffman bracket skein algebra of a closed or one-holed oriented surface of genus $2$.
  This is the first such result for genus greater than $1$.

\end{abstract}

\section{Introduction}

Let $R$ be a commutative ring with identity and with a fixed invertible element $q^{\frac{1}{2}}$.
For an oriented $3$-manifold $M$, its {\it Kauffman bracket skein module} over $R$, denoted by $\mathcal{S}(M;R)$, is defined as the quotient of the free $R$-module generated by isotopy classes of (possibly empty) framed links embedded in $M$ by the submodule generated by the following {\it skein relations}:
\begin{figure}[h]
  \centering
  \includegraphics[width=9cm]{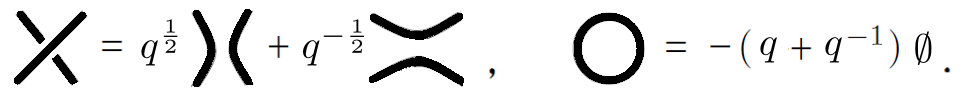}\\
\end{figure}
\\
As a convention, $R$ is identified with $R\emptyset\subset\mathcal{S}(M;R)$ via $\mu\mapsto \mu\emptyset$.

For an oriented surface $\Sigma$, denote $\mathcal{S}(\Sigma\times[0,1];R)$ by $\mathcal{S}(\Sigma;R)$ and call it the {\it Kauffman bracket skein algebra} of $\Sigma$.
Given links $\mathbf{l}_1,\mathbf{l}_2\subset\Sigma\times(0,1)$, the product $\mathbf{l}_1\mathbf{l}_2$ is defined by stacking $\mathbf{l}_1$ over $\mathbf{l}_2$ in the $[0,1]$ direction. Adopt the convention that each framed link is presented as a link equipped with the blackboard framing, i.e. each framing vector points vertically upward.

Throughout the paper, $R=\mathbb{Z}[q^{\pm\frac{1}{2}}]$.
Abbreviate $\mathcal{S}(M;\mathbb{Z}[q^{\pm\frac{1}{2}}])$ to $\mathcal{S}(M)$.

Problem 1.92 (J) in \cite{Ki97} asks for finding the structure of $\mathcal{S}(\Sigma)$ for all surfaces $\Sigma$.
Let $\Sigma_{g,k}$ denote an orientable surface of genus $g$ with $k$ boundary components.
A presentation of $\mathcal{S}(\Sigma_{g,k})$ for $g=0,k\le 4$ and $g=1,k\le 2$ was early given by Bullock and Przytycki \cite{BP00} in 2000.
Recently, a presentation of $\mathcal{S}(\Sigma_{0,5})$ was obtained by Cooke and Lacabanne \cite{CL26} and by Chen \cite{Ch25}.

In this paper, we settle the problem for $\Sigma_{2,1}$ and $\Sigma_{2,0}$. This is the first presentation for a surface of genus $2$. Previously, several structural properties of $\mathcal{S}(\Sigma_{2,0})$ were established in \cite{Ar25,CS21}.

Display $\Sigma_{2,1}$ as in the left part of Figure \ref{fig:Sigma}. Cutting $\Sigma_{2,1}$ along the dotted lines results in a disk $D$.
It is more convenient to view $\Sigma_{2,1}$ as $D$, with four pairs of oriented arcs on $\partial D$ chosen, such that the two arcs in each pair
are identified according to the orientations (see the right part of Figure \ref{fig:Sigma}).
We use arcs in $D$ of certain kinds to stand for loops in $\Sigma_{2,1}$.

\begin{figure}[H]
  \centering
  \includegraphics[width=11.3cm]{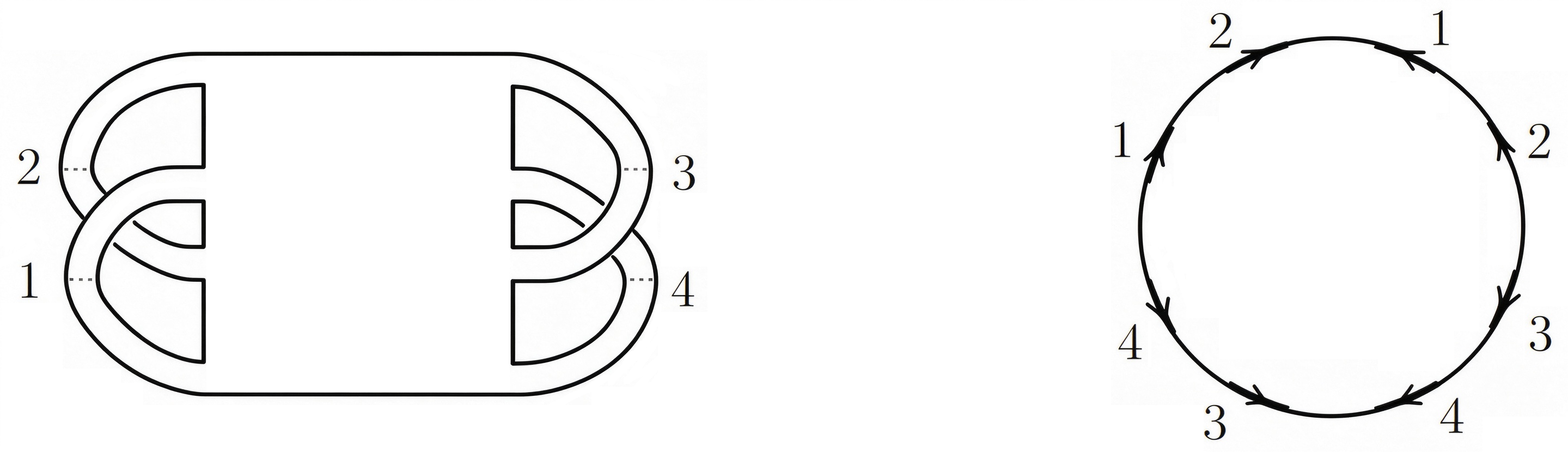}\\
  \caption{Left: the surface $\Sigma_{2,1}$, where the dotted line segment labeled by $k$ is $\mathbf{z}_k$, $1\le k\le 4$. We orient $\mathbf{z}_1,\mathbf{z}_2$ from left to right, and orient $\mathbf{z}_3,\mathbf{z}_4$ from right to left. Right: the disk obtained by cutting $\Sigma_{2,1}$ along the $\mathbf{z}_i$'s; any two arcs with the same label are identified. Below, these oriented arcs will be shown explicitly.}\label{fig:Sigma} 
\end{figure}

\begin{figure}[H]
  \centering
  \includegraphics[width=11.5cm]{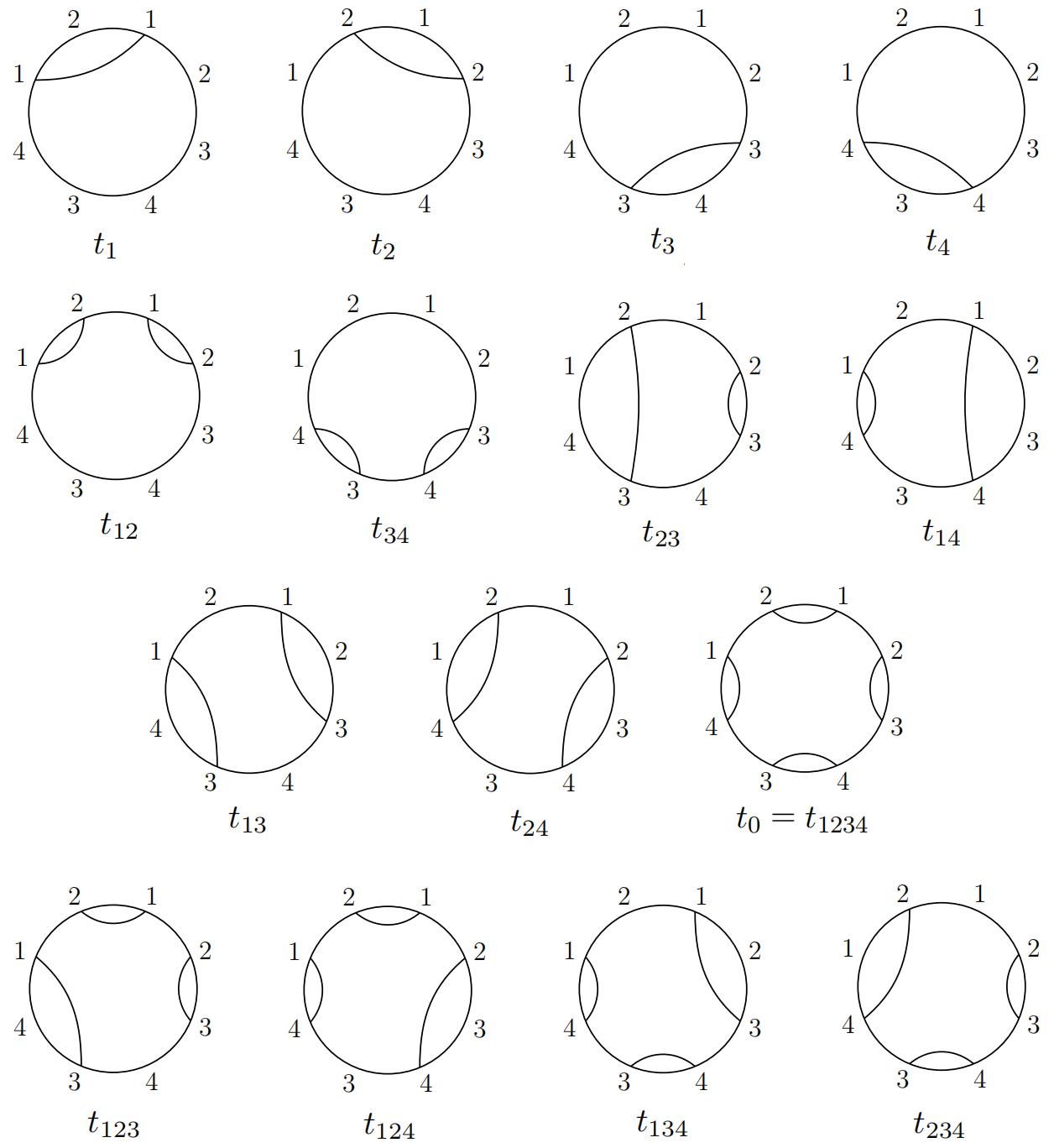}\\
  \caption{$t_{i_1\cdots i_r}$ for $1\le i_1<\cdots<i_r\le 4$.}\label{fig:generator}
\end{figure}

Introduced in Figure \ref{fig:generator} are $15$ elements of $\mathcal{S}(\Sigma_{2,1})$, with $t_0=t_{1234}$. Let
$$\mathcal{G}=\{t_{i_1\cdots i_r}\colon 1\le i_1<\cdots<i_r\le 4\}.$$
Put
\begin{align}
t'_{12}=q^{\frac{1}{2}}t_1t_2-qt_{12},  \qquad   t'_{34}=q^{\frac{1}{2}}t_3t_4-qt_{34}.   \label{eq:t'}
\end{align}

The main results of this paper are the following three theorems.

\begin{thm}\label{thm:presentation-1}
The skein algebra $\mathcal{S}(\Sigma_{2,1})$ is generated by $\mathcal{G}$, and the ideal of defining relations is generated by the following three families.

{\rm(I)} Commuting relations:
\begin{alignat*}{4}
t_1t_3&=t_3t_1, \qquad  &t_1t_4&=t_4t_1, \qquad &t_1t_{34}&=t_{34}t_1, \qquad &t_1t_{14}&=t_{14}t_1, \\
t_1t_{13}&=t_{13}t_1, \qquad  &t_1t_{134}&=t_{134}t_1, \qquad  &t_{12}t_{34}&=t_{34}t_{12}, \qquad  &t_{23}t_{14}&=t_{14}t_{23}, \\
t_{23}t_0&=t_0t_{23}, \quad   &t_{23}t_{134}&=t_{134}t_{23}, \quad  &t_{23}t_{124}&=t_{124}t_{23},  \quad  &t_{123}t_{134}&=t_{134}t_{123},
\end{alignat*}
and the relations obtained from them by symmetry.

{\rm(II)} Commutator relations:
\begin{align}
t_1t_2&=\overline{q}t_2t_1+(q^{\frac{1}{2}}-\overline{q}^{\frac{3}{2}})t_{12},  \hspace{13mm}
t_1t_{12}=qt_{12}t_1+(\overline{q}^{\frac{1}{2}}-q^{\frac{3}{2}})t_2,  \nonumber  \\
t_1t_{23}&=qt_{23}t_1+(\overline{q}^{\frac{1}{2}}-q^{\frac{3}{2}})t_{123},  \hspace{10mm}
t_1t_{24}=qt_{24}t_1+(\overline{q}^{\frac{1}{2}}-q^{\frac{3}{2}})t_{124},  \nonumber   \\
t_1t_0&=\overline{q}t_0t_1+(q^{\frac{1}{2}}-\overline{q}^{\frac{3}{2}})t_{234},  \hspace{10mm}
t_1t_{234}=qt_{234}t_1+(\overline{q}^{\frac{1}{2}}-q^{\frac{3}{2}})t_0,  \nonumber   \\
t_1t_{124}&=\overline{q}t_{124}t_1+(q^{\frac{1}{2}}-\overline{q}^{\frac{3}{2}})t_{24},  \hspace{8.6mm}
t_1t_{123}=\overline{q}t_{123}t_1+(q^{\frac{1}{2}}-\overline{q}^{\frac{3}{2}})t_{23},  \nonumber   \\
t_{12}t_{23}&=qt_{23}t_{12}+(\overline{q}^{\frac{1}{2}}-q^{\frac{3}{2}})t_{13},  \nonumber   \\
t_{12}t_{13}&=\overline{q}t_{13}t_{12}+(q^{\frac{1}{2}}-\overline{q}^{\frac{3}{2}})t_{23}, \nonumber   \\
t_{12}t_0&=t_0t_{12}+(q-\overline{q})(t_2t_{234}-t_1t_{134}),   \label{eq:t12-vs-t0}  \\
t_{12}t_{134}&=\overline{q}t_{134}t_{12}+(q^{\frac{1}{2}}-\overline{q}^{\frac{3}{2}})t_{234}, \nonumber   \\
t_{12}t_{123}&=t_{123}t_{12}+(q-\overline{q})(t_2t_{23}-t_1t_{13}), \nonumber   \\
t_{23}t_{13}&=qt_{13}t_{23}+(\overline{q}^{\frac{1}{2}}-q^{\frac{3}{2}})t_{12},  \nonumber   \\
t_{23}t_{123}&=qt_{123}t_{23}+(\overline{q}^{\frac{1}{2}}-q^{\frac{3}{2}})t_1, \nonumber    \\
t_{13}t_{24}&=t_{24}t_{13}+(q-\overline{q})(t_{12}t_{34}-t_{14}t_{23}), \nonumber   \\
t_{13}t_{234}&=t_{234}t_{13}+(q-\overline{q})(t_{12}t_4-t_{23}t_{134}),  \label{eq:t13-vs-t234} \\
t_{13}t_{123}&=\overline{q}t_{123}t_{13}+(q^{\frac{1}{2}}-\overline{q}^{\frac{3}{2}})t_2,  \nonumber   \\
t_{13}t_0&=\overline{q}^2t_0t_{13}+(q-\overline{q}^3)t_2t_4+(q^{\frac{1}{2}}-\overline{q}^{\frac{3}{2}})(t'_{12}t_{14}+t'_{34}t_{23}) \nonumber \\
&\ \ \ \ +(\overline{q}^{\frac{1}{2}}-q^{\frac{3}{2}})(t_1t_{124}+t_3t_{234})+2(q^2-1)t_{24},  \label{eq:t13-vs-t0}  \\
t_{123}t_{124}&=\overline{q}t_{124}t_{123}+(q^{\frac{1}{2}}-\overline{q}^{\frac{3}{2}})t_{34}, \nonumber   \\
t_{123}t_{234}&=t_{234}t_{123}+(q-\overline{q})(t_1t_4-t_{23}t_0), \nonumber    \\
t_{123}t_0&=\overline{q}t_0t_{123}+(q^{\frac{1}{2}}-\overline{q}^{\frac{3}{2}})t_4,  \label{eq:t123-vs-t0}
\end{align}
and the relations obtained from them by symmetry.

{\rm(III)} Reduction relations:
\begin{align}
t_{13}t_{24}&=qt_{12}t_{34}+\overline{q}t_{14}t_{23}+t_1t_4t_{23}+t_2t_3t_{14}-q^{\frac{1}{2}}(t_1t_{234}+t_3t_{124}) \nonumber \\
&\ \ \ \ -\overline{q}^{\frac{1}{2}}(t_2t_{134}+t_4t_{123})+2t_0,  \label{eq:t13t24}  \\
t_{13}t_{234}&=q^{\frac{1}{2}}t_3t_0+\overline{q}t_{23}t_{134}-\overline{q}^{\frac{1}{2}}t'_{34}t_{123}-qt_1t_3t_{234}+t_1t'_{34}t_{23}-2qt_{124}  \nonumber  \\
&\ \ \ \ +q^{\frac{1}{2}}t_2t_{14}+q^{\frac{3}{2}}t_1t_{24}+qt_4t_{12}, \label{eq:t13t234} \\
t_{123}^2&=(q^{\frac{1}{2}}t_1t_{23}+\overline{q}^{\frac{1}{2}}t_2t_{13}+t_3t_{12})t_{123}
-\overline{q}^{\frac{1}{2}}t'_{12}t_{23}t_{13}+\overline{q}t'_{12}t_{12}-\overline{q}t_1t_3t_{13}   \nonumber \\
&\ \ \ \ -qt_2t_3t_{23}-\overline{q}^2t_{13}^2-q^2t_{23}^2-t_1^2-t_2^2-t_3^2+\alpha^2,  \label{eq:t123t123} \\
t_{123}t_{124}&=\overline{q}^{\frac{1}{2}}t_{12}t_0+\overline{q}^{\frac{3}{2}}(t_1t_{23}t_{124}+t_2t_{14}t_{123}
-t_{14}t'_{12}t_{23}-t_1t_{134})  \nonumber  \\
&\ \ \ \ -q^{\frac{1}{2}}t_2t_{234}-\overline{q}t_{24}t_{23}-\overline{q}^2t_{13}t_{14}+2q^{\frac{1}{2}}t_{34}-t_3t_4,  \label{eq:t123t124} \\
t_{123}t_{134}&=qt_{13}t_0+q^{\frac{5}{2}}(t_1t_{124}+t_3t_{234})-q^{\frac{3}{2}}(t'_{12}t_{14}+t'_{34}t_{23})-2q^3t_{24}-q^2t_2t_4,  \label{eq:t123t134}  \\
t_{123}t_{234}&=(\overline{q}t_{23}+t_2t_3)t_0+t'_{12}t'_{34}t_{23}-qt_3t'_{12}t_{234}-\overline{q}t_2t'_{34}t_{123}-q^{\frac{1}{2}}t_2t_{124}  \nonumber  \\
&\ \ \ \ -\overline{q}^{\frac{1}{2}}t_3t_{134}+q^{\frac{3}{2}}t'_{12}t_{24}+\overline{q}^{\frac{3}{2}}t'_{34}t_{13}
+2t_{14}+qt_1t_4,  \label{eq:t123t234}  \\
t_{123}t_0&=(q^{\frac{1}{2}}t_1t_{23}+t_3t_{12}+\overline{q}^{\frac{1}{2}}t_2t_{13})t_0
-(\overline{q}^{\frac{1}{2}}t'_{12}t_{23}+\overline{q}^2t_{13}+\overline{q}t_1t_3)t_{134}   \nonumber  \\
&\ \ \ \ -q^2t_{23}t_{234}+q^{\frac{3}{2}}t'_{12}t_{124}-t_2(t'_{12}t_{14}+t'_{34}t_{23}+q^{\frac{1}{2}}t_2t_4+q^{\frac{3}{2}}t_{24}) \nonumber \\
&\ \ \ \ -t_3t'_{34}+\overline{q}^{\frac{1}{2}}t_1t_{14}+q^{\frac{3}{2}}\alpha t_4,   \label{eq:t123t0}   \\
t_0^2&=(q^{\frac{1}{2}}t_1t_{234}+\overline{q}^{\frac{1}{2}}t_4t_{123}-\overline{q}t_{14}t_{23}-t_1t_4t_{23})t_0+t_{14}t_{123}t_{234} \nonumber \\
&\ \ \ \ +(\overline{q}^{\frac{3}{2}}t_1t_{23}-\overline{q}^2t_{123})t_{123}+(q^{\frac{3}{2}}t_4t_{23}-q^2t_{234})t_{234}-qt_1t_4t_{14} \nonumber  \\
&\ \ \ \ -t_{14}^2-t_{23}^2-\overline{q}^2t_1^2-q^2t_4^2+\alpha^2,  \label{eq:t0t0}
\end{align}
and the relations obtained from them by symmetry.
\end{thm}

See Convention \ref{conv:symmetry} for the precise meaning of ``by symmetry".
The fact that $\mathcal{G}$ generates $\mathcal{S}(\Sigma_{2,1})$ is a special case of \cite[Theorem 1]{Bu99}. We will state it as Proposition \ref{prop:generate} and reprove it.

\begin{thm}\label{thm:basis-1}
The $R$-module $\mathcal{S}(\Sigma_{2,1})$ is freely generated by $\mathcal{C}$, with
$$\mathcal{C}=\big\{t_1^{i_1}t_2^{i_2}t_3^{i_3}t_4^{i_4}at_{12}^{j_1}t_{23}^{j_2}t_{34}^{j_3}t_{14}^{j_4}\colon i_1,\ldots,i_4,j_1,\ldots,j_4\ge 0,\ a\in\mathcal{A}\big\},$$
where $\mathcal{A}$ consists of $1,t_0,t_{123},t_{124},t_{134},t_{234}$ and
\begin{align*}
t_{13}^k,\ \ t_{13}^kt_0,\ \ t_{13}^kt_{123},\ \  t_{13}^kt_{134}, \ \ t_{24}^k, \ \ t_{24}^kt_0, \ \ t_{24}^kt_{124}, \ \ t_{24}^kt_{234} \ \ \ \text{for\ \ }k\ge 1.
\end{align*}
\end{thm}
This gives a monomial basis for $\mathcal{S}(\Sigma_{2,1})$. It is similar to the one given in \cite[Theorem 1.3]{Ch25}, but the proof is not a straightforward repetition. Instead, some technical issues caused by $1$-handles need to be addressed carefully.

\begin{thm}\label{thm:presentation-2}
The skein algebra $\mathcal{S}(\Sigma_{2,0})$ is generated by $\mathcal{G}$, and the ideal of defining relations is generated by the three families of relations given in Theorem \ref{thm:presentation-1} together with
\begin{align}
t_{12}t'_{12}-t_{34}t'_{34}&=\overline{q}t_1^2+qt_2^2-\overline{q}t_3^2-qt_4^2,  \label{eq:Sl-0}   \\
t_{34}t_{134}&=qt_4t_{14}+\overline{q}t_3t_{13}+2t_1, \label{eq:Sl-1}  \\
t_{34}t_{234}&=qt_4t_{24}+\overline{q}t_3t_{23}+2t_2, \label{eq:Sl-2}  \\
t_{12}t_{123}&=qt_2t_{23}+\overline{q}t_1t_{13}+2t_3, \label{eq:Sl-3}  \\
t_{12}t_{124}&=qt_2t_{24}+\overline{q}t_1t_{14}+2t_4, \label{eq:Sl-4}  \\
t_4t_{134}&=t_2t_{123}+q(t'_{34}t_{14}-t'_{12}t_{23}),  \label{eq:Sl-13}  \\
t_3t_{234}&=t_1t_{124}+\overline{q}(t'_{34}t_{23}-t'_{12}t_{14}),  \label{eq:Sl-24}  \\
t_1t_{234}-t_3t_{124}&=\overline{q}(t_2t_{134}-t_4t_{123})+q^{\frac{1}{2}}(t_1t_4t_{23}-t_2t_3t_{14}),  \label{eq:Sl-1234}  \\
t_{34}t'_{34}t_{14}&=(\overline{q}t_3^2+qt_4^2)t_{14}+\overline{q}^{\frac{1}{2}}(t'_{34}t_{13}-t_2t_{124}) \nonumber \\
&\hspace{10mm} +q^{\frac{1}{2}}(t'_{12}t_{24}-t_3t_{134})+2t_1t_4,  \label{eq:Sl-14}   \\
t_{12}t'_{12}t_{23}&=(\overline{q}t_1^2+qt_2^2)t_{23}+\overline{q}^{\frac{1}{2}}(t'_{12}t_{13}-t_4t_{234}) \nonumber \\
&\hspace{10mm} +q^{\frac{1}{2}}(t'_{34}t_{24}-t_1t_{123})+2t_2t_3,  \label{eq:Sl-23}  \\
t_{34}t_0&=qt_4t_{124}+\overline{q}t_3t_{123}+2t'_{12}, \label{eq:Sl-12}  \\
t_{12}t_0&=qt_2t_{234}+\overline{q}t_1t_{134}+2t'_{34}, \label{eq:Sl-34}  \\
t_4t_0&=qt'_{34}t_{124}-q^{\frac{1}{2}}t_2t'_{34}t_{14}+\overline{q}^{\frac{1}{2}}t_2t_4t_{134}-\overline{q}t_2t_{13}+t_1t_{23},
\label{eq:Sl-123}  \\
t_3t_0&=\overline{q}t'_{34}t_{123}-\overline{q}^{\frac{1}{2}}t_1t'_{34}t_{23}+q^{\frac{1}{2}}t_1t_3t_{234}-qt_1t_{24}+t_2t_{14},
\label{eq:Sl-124}  \\
t_2t_0&=qt'_{12}t_{234}-q^{\frac{1}{2}}t_4t'_{12}t_{23}+\overline{q}^{\frac{1}{2}}t_2t_4t_{123}-\overline{q}t_4t_{13}+t_3t_{14},
\label{eq:Sl-134}  \\
t_1t_0&=\overline{q}t'_{12}t_{134}-\overline{q}^{\frac{1}{2}}t_3t'_{12}t_{14}+q^{\frac{1}{2}}t_1t_3t_{124}-qt_3t_{24}+t_4t_{23}.
\label{eq:Sl-234}
\end{align}
\end{thm}

The remainder of this paper is structured as follows. In Section 2 we prove Theorem \ref{thm:presentation-1} and Theorem \ref{thm:basis-1}.
The strategy is first to verify the relations given in Theorem \ref{thm:presentation-1} directly, and then to show that using these relations, each element of $\mathcal{S}(\Sigma_{2,1})$ can be transformed into an $R$-linear combination of elements of $\mathcal{C}$, and that $\mathcal{C}$ is $R$-linearly independent. In Section 3 we prove Theorem \ref{thm:presentation-2}, based on the fact that $\Sigma_{2,0}$ can be obtained by attaching a 2-handle to $\Sigma_{2,1}$. Section 4 is an appendix containing many pictorial formulas.

\begin{nota}
\rm When $\mathcal{M}$ is a left module over a ring $\mathcal{R}$, and $\mathcal{X}$ is a subset of $\mathcal{M}$, we write $\mathcal{M}=\mathcal{R}\langle\mathcal{X}\rangle$ to indicate that $\mathcal{X}$ generates $\mathcal{M}$.

Write $\overline{q}=q^{-1}$, $\overline{q}^{\frac{1}{2}}=q^{-\frac{1}{2}}$, etc. Let $\alpha=q+\overline{q}$.

In most places, we abbreviate $\Sigma_{2,1}$ to $\Sigma$.

For a finite set $X$, let $\#X$ denote its cardinality.
\end{nota}

\noindent
{\bf AI-use declaration}: DeepSeek was used for literature search and language polishing. The author takes full responsibility for all mathematical content.

\section{The skein algebra of $\Sigma_{2,1}$}

\subsection{Fundamentals}

Let $Z=\sqcup_{k=1}^4Z_k$, with $Z_k=\mathbf{z}_k\times[0,1]$.
Let $\pi:\Sigma\times[0,1]\to\Sigma$ denote the projection.
Let $F_4$ denote the free group on the letters $x_1,\ldots,x_4$.

\begin{defn}
\rm For a $1$-submanifold $\mathbf{x}\subset\Sigma\times[0,1]$, which is always assumed to intersect $Z$ transversely, let $|\mathbf{x}|_k=\#(\mathbf{x}\cap Z_k)$, $|\mathbf{x}|=\sum_{k=1}^4|\mathbf{x}|_k$, and let
${\rm supp}(\mathbf{x})=\{1\le k\le 4\colon |\mathbf{x}|_k>0\}$.
Call $\mathbf{x}$ {\it unshortenable} if $|\mathbf{x}|_k\le 1$ for all $k$. Call $\mathbf{x}$ {\it simple} if it has no crossing, i.e. if $\pi(\mathbf{x})$ has no self-intersection.
\end{defn}

Given an oriented arc $\mathbf{a}\subset\Sigma\times[0,1]$, define ${\rm word}(\mathbf{a})\in F_4$ as follows: walking along $\pi(\mathbf{a})$ guided by the orientation inherited from $\mathbf{a}$, and record $x_k$ (resp. $x_k^{-1}$) whenever it passes through $\mathbf{z}_k$ from right (resp. left); let ${\rm word}(\mathbf{a})$ be the product of all recorded letters. 

\begin{defn}
\rm Call an arc $\mathbf{a}$ {\it minimal shortenable} if, for either choice of an orientation,
${\rm word}(\mathbf{a})=x_{j_1}^{\epsilon_1}\cdots x_{j_n}^{\epsilon_n}$, with $\epsilon_1,\ldots,\epsilon_n\in\{\pm1\}$ such that $j_1,\ldots,j_{n-1}$ are distinct and $j_n=j_1$.
\end{defn}

Let $\mathcal{T}$ denote the $R$-subalgebra of $\mathcal{S}(\Sigma)$ generated by unshortenable simple loops.

Given $\mathsf{x},\mathsf{y}\in\Sigma\times\{0\}$, let $\mathcal{E}_{\mathsf{x},\mathsf{y}}$ denote the $R$-module generated by
$\partial$-relative isotopy classes of 1-submanifolds $\mathbf{x}\subset\Sigma\times[0,1]$ with $\mathbf{x}\cap(\Sigma\times\{0\})=\partial\mathbf{x}=\{\mathsf{x},\mathsf{y}\},$
modulo skein relations. By stacking, $\mathcal{E}_{\mathsf{x},\mathsf{y}}$ is a left $\mathcal{S}(\Sigma)$-module.

\begin{lem}\label{lem:chop-up}
For each minimal shortenable arc $\mathbf{a}$ representing an element of $\mathcal{E}_{\mathsf{x},\mathsf{y}}$, there exist $a_i\in\mathcal{T}$
and unshortenable arcs $\mathbf{c}_i\in\mathcal{E}_{\mathsf{x},\mathsf{y}}$ with $|\mathbf{c}_i|<|\mathbf{a}|$ and $\mathbf{a}=\sum_ia_i\mathbf{c}_i$ in $\mathcal{E}_{\mathsf{x},\mathsf{y}}$.
\end{lem}

\begin{proof}
If $\mathbf{a}$ is nonsimple, then by resolving the crossings, we may write $\mathbf{a}=\sum_ib_i\mathbf{a}_i$, with $b_i\in\mathcal{T}$ and $\mathbf{a}_i$ simple; furthermore, each $\mathbf{a}_i$ is unshortenable or minimal shortenable. Hence we may assume that $\mathbf{a}$ is simple.

When one of the two orientations of $\mathbf{a}$ is chosen, ${\rm word}(\mathbf{a})=x_{j_1}^{\epsilon_1}\cdots x_{j_n}^{\epsilon_n}$, with $j_1,\ldots,j_{n-1}$ distinct, $\epsilon_1=1$ and $j_n=j_1$. Without loss of generality, we may assume $j_1=j_n=1$.
The cases $\epsilon_n=1$ and $\epsilon_n=-1$ are illustrated in the upper and lower parts of Figure \ref{fig:chop}, respectively. In either case, the underlined term is $\mathbf{a}$, and each of the remaining terms has the form $b\mathbf{c}$, where $b\in\mathcal{T}$ and $\mathbf{c}$ is unshortenable with $|\mathbf{c}|<|\mathbf{a}|$. The assertion is clear.

\end{proof}

For convenience, we introduce some auxiliary elements in Figure \ref{fig:element-1}.
As we will see, the notations for the first two elements are consistent with (\ref{eq:t'}).

\begin{lem}\label{lem:generating-T}
$\mathcal{T}$ is generated by $\mathcal{G}$.
\end{lem}

\begin{figure}[H]
  \centering
  \includegraphics[width=13cm]{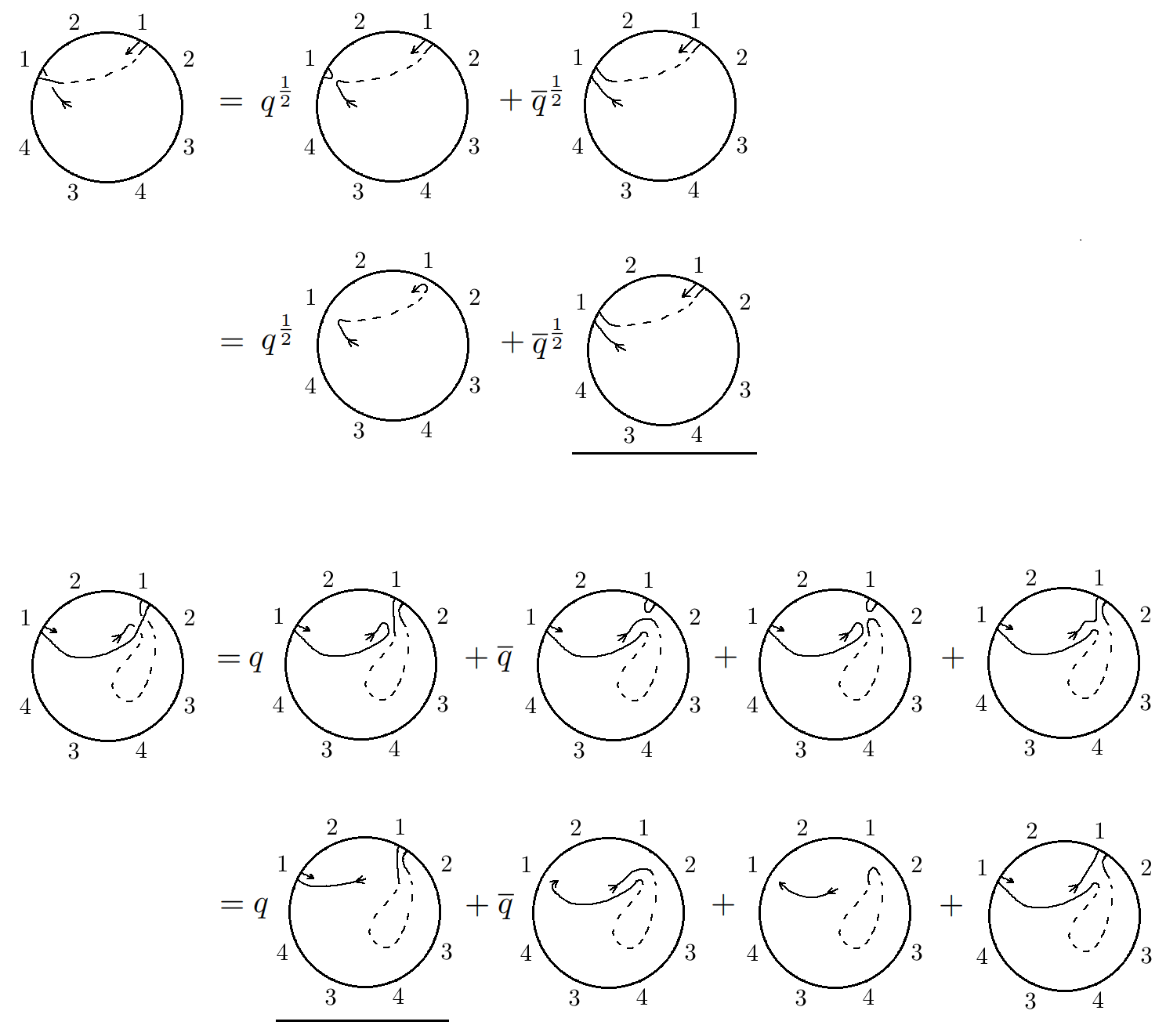}\\
  \caption{The dotted arcs indicate the irrelevant parts, which may be complicated.}\label{fig:chop}
\end{figure}

\begin{figure}[H]
  \centering
  \includegraphics[width=12.3cm]{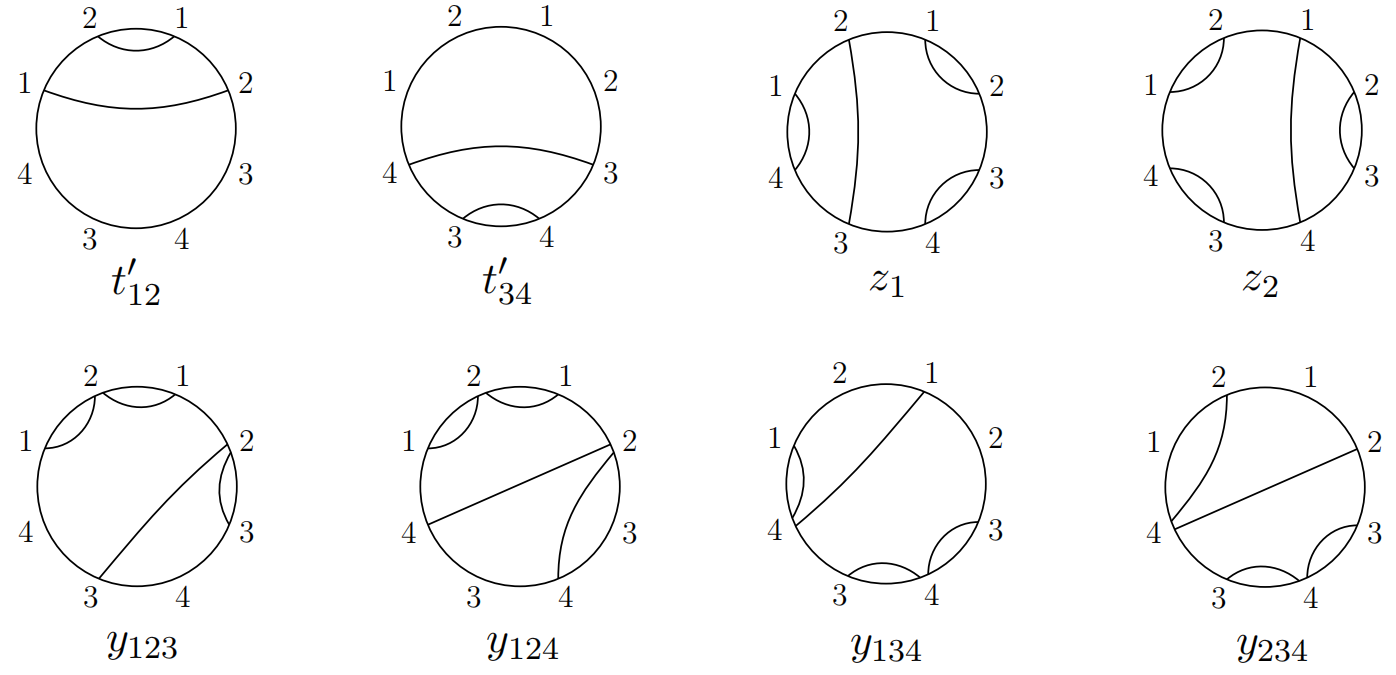}\\
  \caption{}\label{fig:element-1}
\end{figure}

\newpage

\begin{proof}
Define an equivalence relation among unshortenable simple loops by declaring $\mathbf{s}\sim\mathbf{s}'$ if $\mathbf{s}'=\beta\mathbf{s}+\sum_i\gamma_i\mathbf{a}_i$ for some $\beta\in R^\times$, $\gamma_i\in R$, and $\mathbf{a}_i$ being a product of loops of degree smaller than $|\mathbf{s}|$.

It suffices to show that each unshortenable simple loop $\mathbf{s}$ is equivalent to a unique $\mathbf{s}_0\in\mathcal{G}$ with
${\rm supp}(\mathbf{s}_0)={\rm supp}(\mathbf{s})$.

If ${\rm supp}(\mathbf{s})=\{1,2\}$, then $\mathbf{s}\in\{t_{12},t'_{12}\}$; by either of the two equations in Figure \ref{fig:t1-vs-t2}, $t'_{12}\sim t_{12}$. The situation is similar when ${\rm supp}(\mathbf{s})=\{3,4\}$.
If ${\rm supp}(\mathbf{s})=\{i,j\}$ with $(i,j)\in\{(1,3),(1,4),(2,3),(2,4)\}$, then $\mathbf{s}=t_{ij}$.

When $|\mathbf{s}|=3$, up to symmetry we may assume ${\rm supp}(\mathbf{s})=\{1,2,3\}$. It is easy to see that $\mathbf{s}\in\{t_{123},t'_{123},t''_{123}\}$, where $t'_{123}$ and $t''_{123}$ denote, respectively, the leftmost terms in the upper and lower equations in Figure \ref{fig:element-2}; by Figure \ref{fig:element-2}, $t'_{123},t''_{123}\sim t_{123}$.

When ${\rm supp}(\mathbf{s})=\{1,2,3,4\}$, there are $9$ possibilities as illustrated in Figure \ref{fig:equivalence-1}; any two adjacent loops are equivalent for a reason similar to that for $t'_{123}\sim t_{123}$. In particular, $\mathbf{s}\sim t_{1234}$.

This completes the proof.
\end{proof}

\begin{figure}
  \centering
  \includegraphics[width=8.5cm]{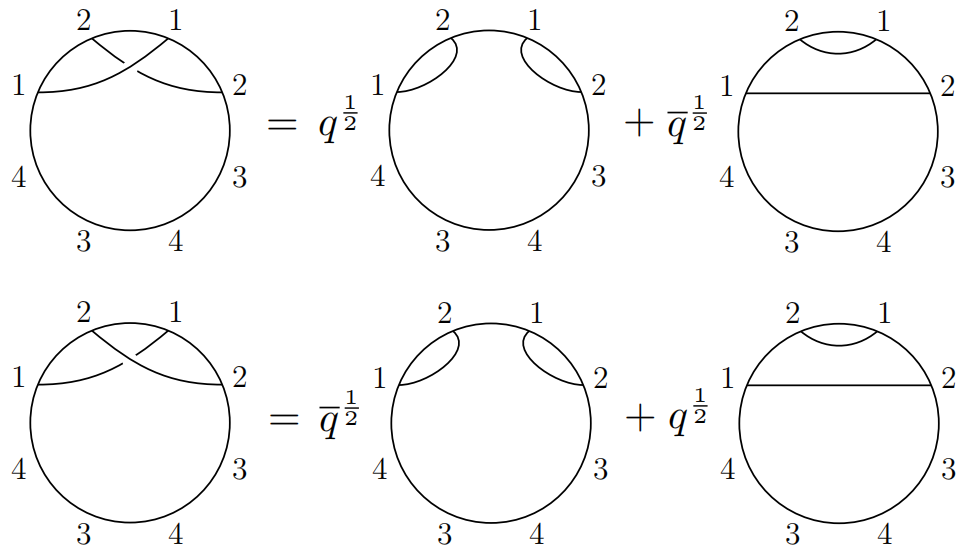}\\
  \caption{}\label{fig:t1-vs-t2}
\end{figure}

\begin{figure}
  \centering
  \includegraphics[width=8.5cm]{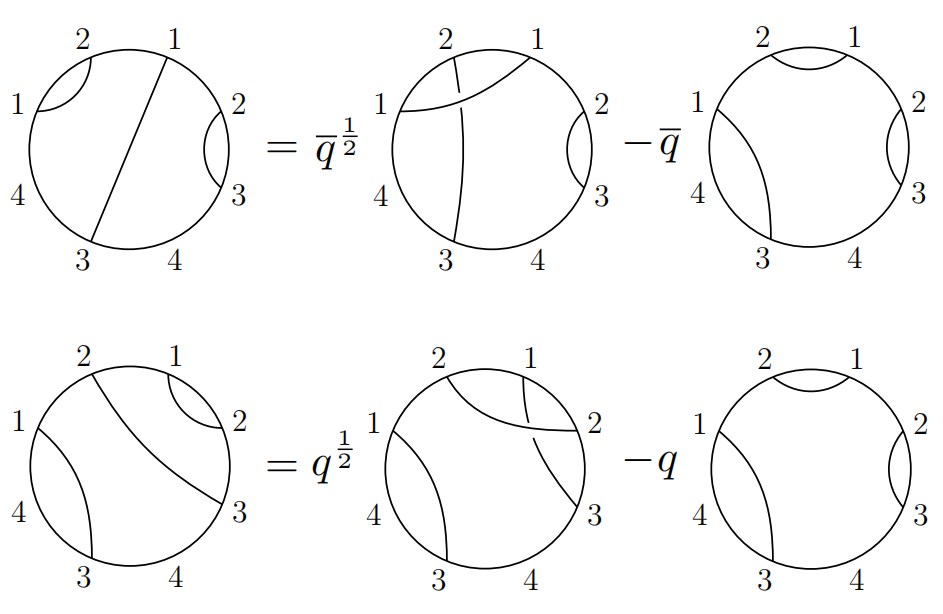}\\
  \caption{}\label{fig:element-2}
\end{figure}

\begin{figure}
  \centering
  \includegraphics[width=10cm]{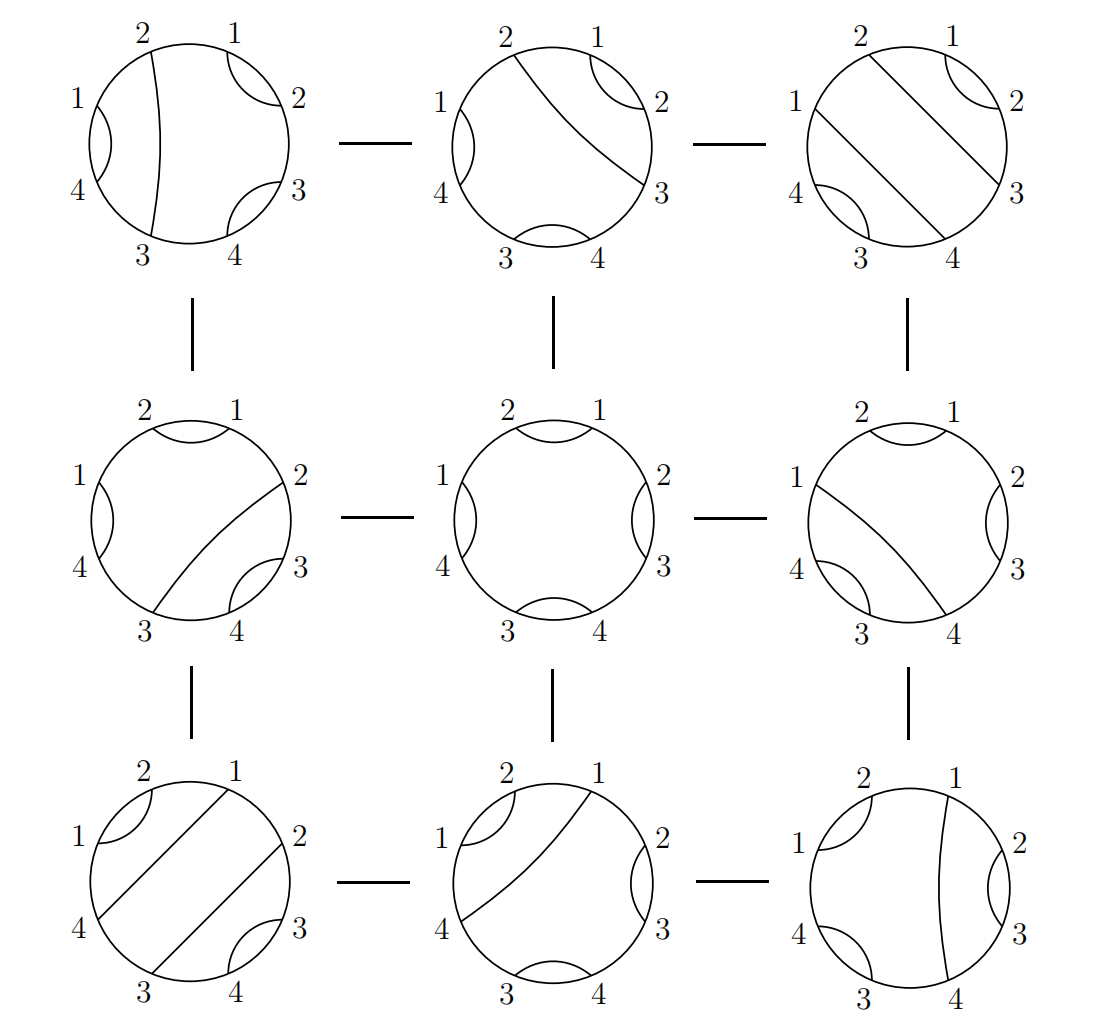}\\
  \caption{}\label{fig:equivalence-1}
\end{figure}

\begin{prop}\label{prop:generate}
$\mathcal{S}(\Sigma)$ is generated by $\mathcal{G}$.
\end{prop}

\begin{proof}
By Lemma \ref{lem:generating-T}, it suffices to show that $\mathbf{l}\in\mathcal{T}$ for each link $\mathbf{l}$.
By resolving the crossings, we can write $\mathbf{l}$ as an $R$-linear combination of multicurves. Since each multicurve equals a product of simple loops, it is sufficient to show that each simple loop $\mathbf{s}$ belongs to $\mathcal{T}$.

If $\mathbf{s}$ is unshortenable, then $\mathbf{s}\in\mathcal{T}$ already. Otherwise, $\mathbf{s}$ contains a minimal shortenable arc $\mathbf{b}$, which, by Lemma \ref{lem:chop-up}, can be replaced by $\sum_ia_i\mathbf{c}_i$ for some $a_i\in\mathcal{T}$ and arcs $\mathbf{c}_i$ with $|\mathbf{c}_i|<|\mathbf{b}|$. Accordingly, $\mathbf{s}$ is replaced by $\sum_ia_i\mathbf{k}_i$, for some knots $\mathbf{k}_i$ with $|\mathbf{k}_i|<|\mathbf{s}|$. Next, for each $i$, we resolve the crossings of $\mathbf{k}_i$ to write $\mathbf{k}_i$ as an $R$-linear combination of products of simple loops. By repeating this process, we can ultimately write $\mathbf{s}$ as a (noncommutative) polynomial in unshortenable simple loops. Therefore, $\mathbf{s}\in\mathcal{T}$.
\end{proof}

\begin{figure}
  \centering
  \includegraphics[width=12cm]{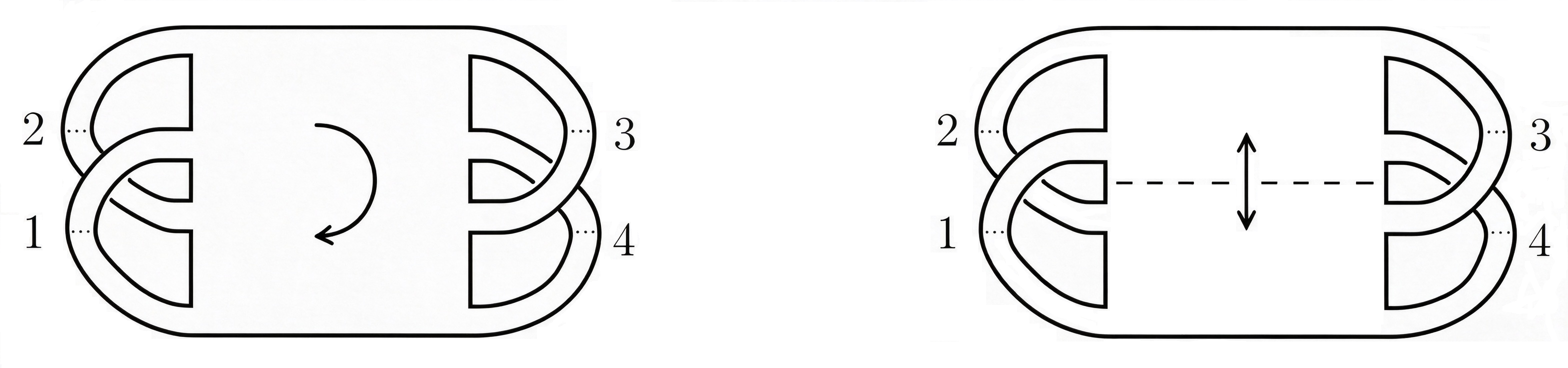}\\
  \caption{Left: the rotation $\sigma$. Right: the reflection $\tau$.}\label{fig:symmetry}
\end{figure}

Let $\sigma$ denote the rotation of $\Sigma$ through an angle $\pi$, and let $\tau$ denote the reflection of $\Sigma$ in the horizontal dotted line.
See Figure \ref{fig:symmetry} for illustration.

Clearly, $\sigma$ induces an $R$-algebra involution $\sigma_\ast:\mathcal{S}(\Sigma)\to\mathcal{S}(\Sigma)$.
Since $\tau$ switches the type of each crossing of each link in $\Sigma\times[0,1]$, it induces an involution of abelian groups $\tau_\ast:\mathcal{S}(\Sigma)\to\mathcal{S}(\Sigma)$ such that $\tau_\ast(q\mathbf{l})=\overline{q}\tau_\ast(\mathbf{l})$ and $\tau_\ast(\mathbf{l}\mathbf{l}')=\tau_\ast(\mathbf{l})\tau_\ast(\mathbf{l}')$ for all links $\mathbf{l},\mathbf{l}'$.

The map $\sigma_\ast$ fixes $t_{13},t_{24},t_0$, and permutes the other generators as
\begin{align*}
t_1\leftrightarrow t_3, \quad  t_2\leftrightarrow t_4, \quad  t_{12}\leftrightarrow t_{34}, \quad  t_{23}\leftrightarrow t_{14},  \quad
t_{234}\leftrightarrow t_{124}, \quad t_{134}\leftrightarrow t_{123}.
\end{align*}
The map $\tau_\ast$ fixes $t_{12}, t_{34}, t_0$, and permutes the other generators as
\begin{align*}
t_1\leftrightarrow t_2, \quad  t_3\leftrightarrow t_4, \quad
t_{23}\leftrightarrow t_{14}, \quad  t_{13}\leftrightarrow t_{24}, \quad t_{234}\leftrightarrow t_{134}, \quad  t_{124}\leftrightarrow t_{123}.
\end{align*}

\begin{conv}\label{conv:symmetry}
\rm In what follows, ``by symmetry" means acting via one of $\sigma_\ast$, $\tau_\ast$, and $\sigma_\ast\tau_\ast=\tau_\ast\sigma_\ast$.

We will frequently use equations similar to those in Figure \ref{fig:element-2}.
\end{conv}


\begin{figure}[h]
  \centering
  \includegraphics[width=11cm]{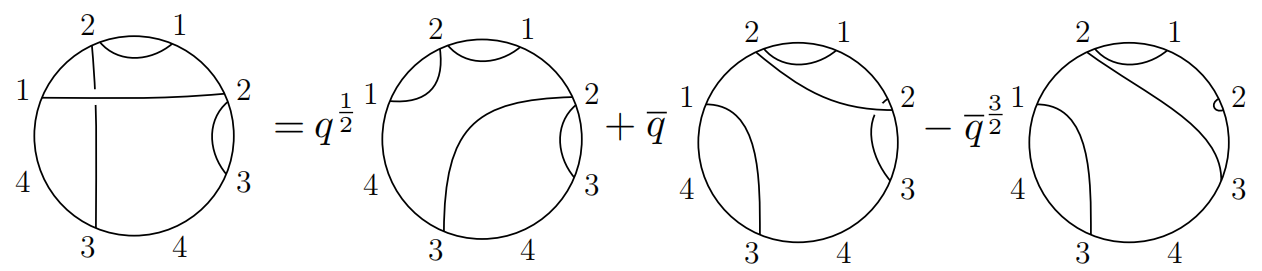}\\
  \caption{$t'_{12}t_{23}=q^{\frac{1}{2}}y_{123}+\overline{q}t_2t_{123}-\overline{q}^{\frac{3}{2}}t_{13}$.}\label{fig:deducing-element-1}
\end{figure}

\begin{figure}[h]
  \centering
  \includegraphics[width=11cm]{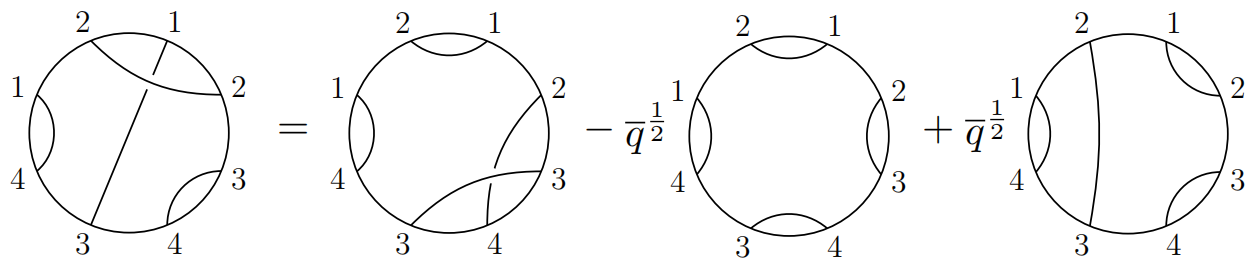}\\
  \caption{$t_2(\overline{q}^{\frac{1}{2}}t_3t_{14}-\overline{q}t_{134})=t_3t_{124}-\overline{q}^{\frac{1}{2}}t_0+\overline{q}^{\frac{1}{2}}z_1$.}
  \label{fig:deducing-element-2}
\end{figure}

By Figure \ref{fig:deducing-element-1},
\begin{align}
y_{123}=\overline{q}^{\frac{1}{2}}t'_{12}t_{23}-\overline{q}^{\frac{3}{2}}t_2t_{123}+\overline{q}^2t_{13}.  \label{eq:y123}
\end{align}
Similarly,
\begin{align}
y_{124}&=\overline{q}^{\frac{1}{2}}t'_{12}t_{24}-\overline{q}^{\frac{3}{2}}t_2t_{124}+\overline{q}^2t_{14},  \nonumber   \\
y_{134}&=\overline{q}^{\frac{1}{2}}t'_{34}t_{14}-\overline{q}^{\frac{3}{2}}t_4t_{134}+\overline{q}^2t_{13},  \label{eq:y134}  \\
y_{234}&=\overline{q}^{\frac{1}{2}}t'_{34}t_{24}-\overline{q}^{\frac{3}{2}}t_4t_{234}+\overline{q}^2t_{23}.  \nonumber
\end{align}

By Figure \ref{fig:deducing-element-2},
\begin{align}
z_1=t_0-\overline{q}^{\frac{1}{2}}t_2t_{134}-q^{\frac{1}{2}}t_3t_{124}+t_2t_3t_{14}.  \label{eq:z1}
\end{align}
By symmetry,
\begin{align}
z_2=t_0-q^{\frac{1}{2}}t_1t_{234}-\overline{q}^{\frac{1}{2}}t_4t_{123}+t_1t_4t_{23}.  \label{eq:z2}
\end{align}

\subsection{Verifying the relations}\label{sec:verify}

The commuting relations are easy to verify, so we omit the details.

\begin{figure}[h]
  \centering
  \includegraphics[width=12.8cm]{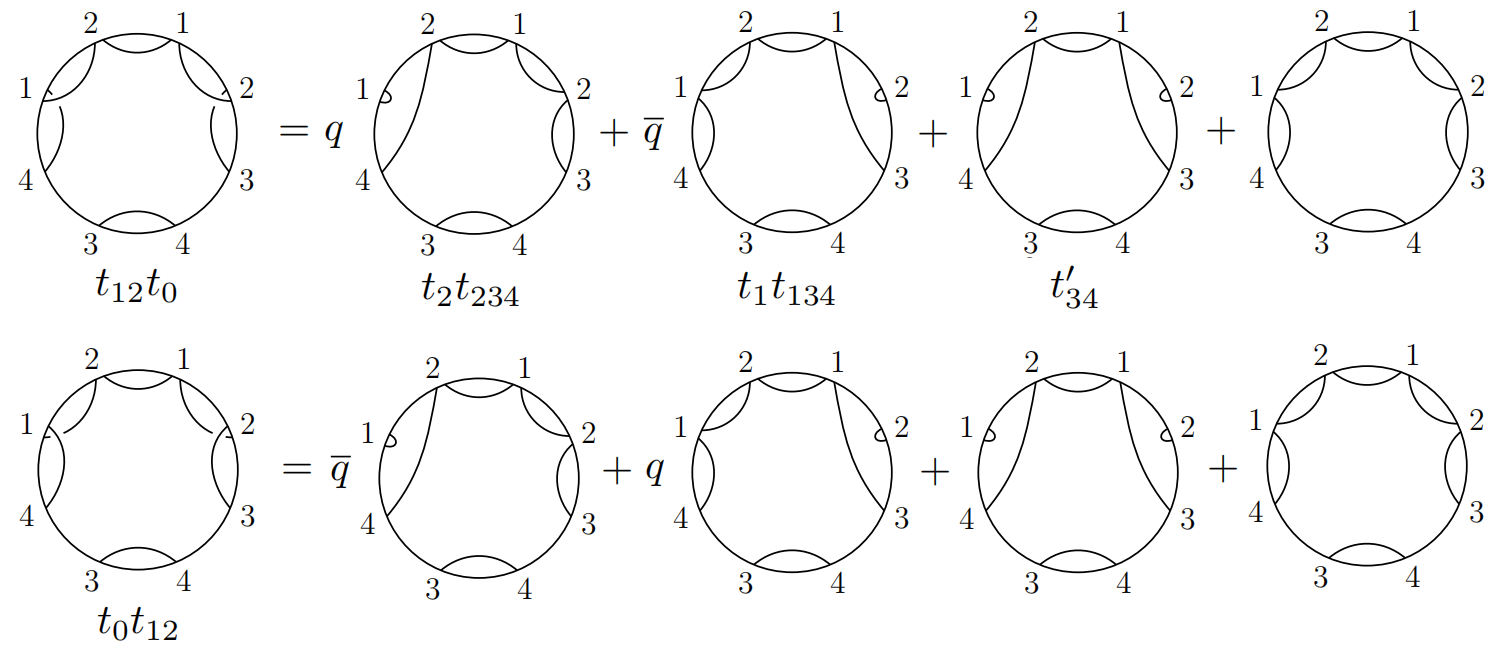}\\
  \caption{Expanding $t_{12}t_0$ and $t_0t_{12}$.}\label{fig:t12-vs-t0}
\end{figure}

\begin{figure}[h]
  \centering
  \includegraphics[width=12.8cm]{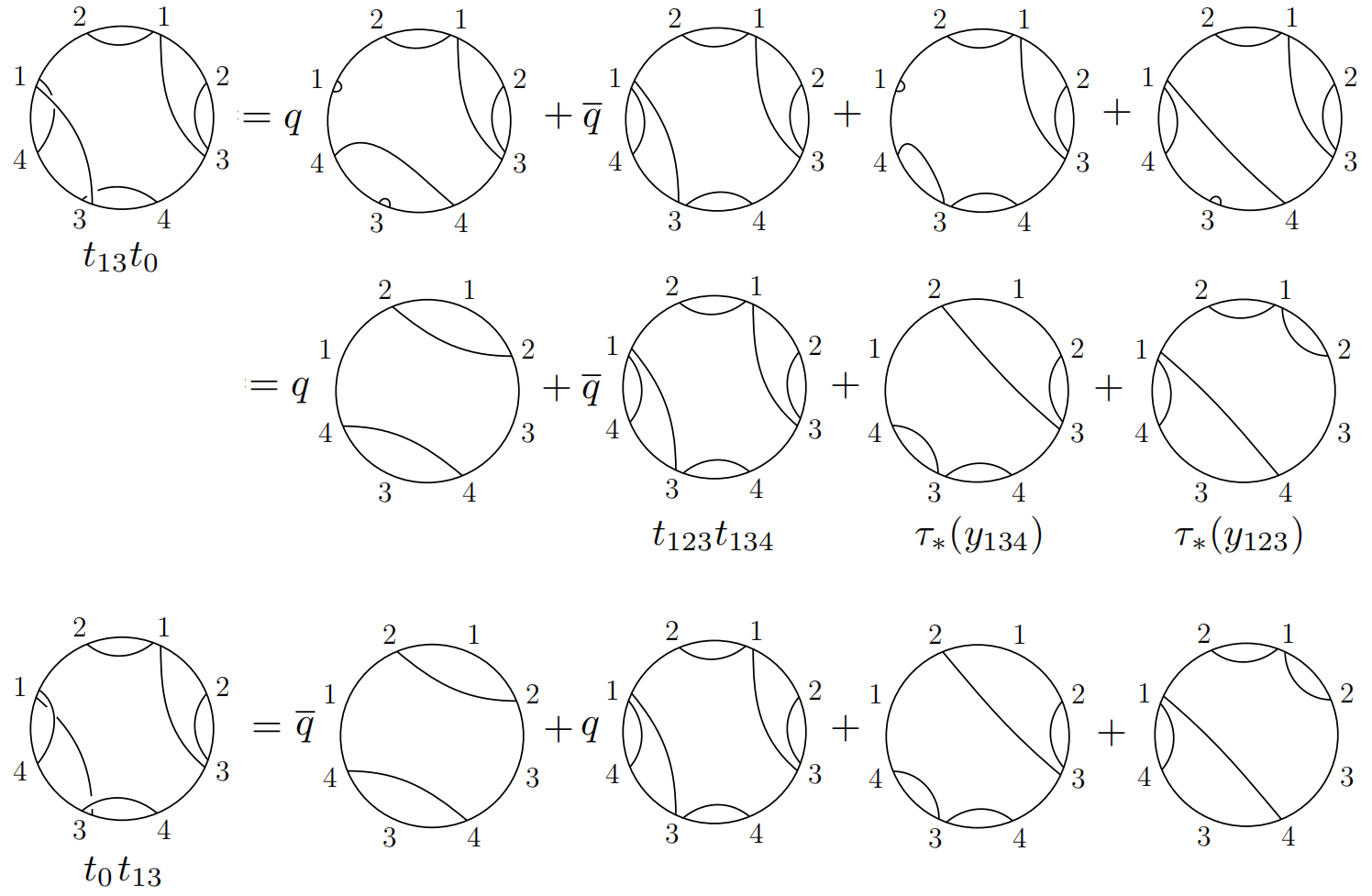}\\
  \caption{Expanding $t_{13}t_0$ and $t_0t_{13}$.}\label{fig:t13-vs-t0}
\end{figure}

\begin{proof}[Proof of commutator relations]
Comparing the two equations in Figure \ref{fig:t1-vs-t2} and eliminating $t'_{12}$, we obtain
$t_1t_2=\overline{q}t_2t_1+(q^{\frac{1}{2}}-\overline{q}^{\frac{3}{2}})t_{12}$.

Taking the difference between the two equations in Figure \ref{fig:t12-vs-t0} yields (\ref{eq:t12-vs-t0}).

Each of the other equations can be deduced in a similar way to one of these two equations, except (\ref{eq:t13-vs-t0}), which requires more explanation.

Comparing the equations in Figure \ref{fig:t13-vs-t0} and eliminating $t_{123}t_{134}$ yields
\begin{align*}
t_{13}t_0=\overline{q}^2t_0t_{13}+(q-\overline{q}^3)t_2t_4+(1-\overline{q}^2)(\tau_\ast(y_{134})+\tau_\ast(y_{123})).
\end{align*}
Acting on (\ref{eq:y134}), (\ref{eq:y123}) by $\tau_\ast$, we obtain
\begin{align*}
\tau_\ast(y_{134})&=q^{\frac{1}{2}}t'_{34}t_{23}-q^{\frac{3}{2}}t_3t_{234}+q^2t_{24},   \\
\tau_\ast(y_{123})&=q^{\frac{1}{2}}t'_{12}t_{14}-q^{\frac{3}{2}}t_1t_{124}+q^2t_{24},
\end{align*}
respectively, and then deduce (\ref{eq:t13-vs-t0}).
\end{proof}

\begin{proof}[Proof of reduction relations]
From Figure \ref{fig:t13-times-t24} we see
$$t_{13}t_{24}=qt_{12}t_{34}+\overline{q}t_{14}t_{23}+z_1+z_2,$$
which, by (\ref{eq:z1}), (\ref{eq:z2}), is just (\ref{eq:t13t24}).

The deduction of (\ref{eq:t13t234}) is shown in Figure \ref{fig:t13-times-t234}.

By Figure \ref{fig:t123-times-t123},
\begin{align*}
t'_{12}t_{23}t_{13}&=(qt_1t_{23}-q^{\frac{1}{2}}t_{123})t_{123}-q^{\frac{5}{2}}t_{23}^2-q^{\frac{1}{2}}t_1^2-q^{\frac{3}{2}}(-\alpha) \\
&\ \ \ \ +(t_2t_{13}-q^{\frac{1}{2}}t_{123})t_{123}-q^{\frac{1}{2}}t_2^2-\overline{q}^{\frac{3}{2}}t_{13}^2-\overline{q}^{\frac{1}{2}}(-\alpha)
+q^{\frac{1}{2}}t_3t_{12}t_{123}   \\
&\ \ \ \ -q^{\frac{3}{2}}t_2t_3t_{23}-\overline{q}^{\frac{1}{2}}t_1t_3t_{13}-q^{\frac{1}{2}}t_3^2+q^{\frac{1}{2}}t_{123}^2
+\overline{q}^{\frac{1}{2}}t'_{12}t_{12},
\end{align*}
which implies (\ref{eq:t123t123}).

By Figure \ref{fig:t123-times-t124},
\begin{align*}
t_{14}t'_{12}t_{23}&=qt_{12}t_0-q^2t_2t_{234}-t_1t_{134}-qt'_{34}+t_1t_{23}t_{124}-q^{\frac{1}{2}}t_{24}t_{23}+t_{34} \\
&\ \ \ \ +t_2t_{14}t_{123}-q^{\frac{1}{2}}t_{124}t_{123}-\overline{q}^{\frac{1}{2}}t_{13}t_{14}  \\
&=qt_{12}t_0-q^2t_2t_{234}-t_1t_{134}+t_1t_{23}t_{124}-q^{\frac{1}{2}}t_{24}t_{23} \\
&\ \ \ \ +t_2t_{14}t_{123}-q^{\frac{3}{2}}t_{123}t_{124}-\overline{q}^{\frac{1}{2}}t_{13}t_{14}+2q^2t_{34}-q^{\frac{3}{2}}t_3t_4,
\end{align*}
which implies (\ref{eq:t123t124}).

Equation (\ref{eq:t123t134}) follows from the upper equation in Figure \ref{fig:t13-vs-t0}, which reads
\begin{align*}
t_{13}t_0&=qt_2t_4+\overline{q}t_{123}t_{134}+\tau_\ast(y_{134})+\tau_\ast(y_{123}) \\
&=qt_2t_4+\overline{q}t_{123}t_{134}+(q^{\frac{1}{2}}t'_{34}t_{23}-q^{\frac{3}{2}}t_3t_{234}+q^2t_{24})  \\
&\ \ \ \ +(q^{\frac{1}{2}}t'_{12}t_{14}-q^{\frac{3}{2}}t_1t_{124}+q^2t_{24}).
\end{align*}

Equation (\ref{eq:t123t234}) can be obtained by comparing the formulas in Figure \ref{fig:t123-times-t234}.

By Figure \ref{fig:t123-times-t0},
\begin{align*}
t_1t_{23}t_0&=q(q^{\frac{1}{2}}t_{23}t_{234}-qt_4)+\overline{q}^{\frac{1}{2}}t_{123}t_0+\overline{q}^{\frac{1}{2}}y_{123}t_{134}
-\overline{q}\big(q^{\frac{1}{2}}t_3t_{12}t_0-q^{\frac{3}{2}}t_2t_3t_{234}   \\
&\ \ \ \ -\overline{q}^{\frac{1}{2}}t_1t_3t_{134}-q^{\frac{1}{2}}t_3t'_{34}-qt_{12}t_{124}+q^2t_2t_{24}+t_1t_{14}+qt_4\big).
\end{align*}
With (\ref{eq:y123}) substituted, this implies
\begin{align*}
t_{123}t_0&=(q^{\frac{1}{2}}t_1t_{23}+t_3t_{12})t_0-(\overline{q}^{\frac{1}{2}}t'_{12}t_{23}-\overline{q}^{\frac{3}{2}}t_2t_{123}
+\overline{q}^2t_{13})t_{134}-\overline{q}t_1t_3t_{134}   \\
&\ \ \ \ -(q^2t_{23}+qt_2t_3)t_{234}-t_3t'_{34}-q^{\frac{1}{2}}t_{12}t_{124}+q^{\frac{3}{2}}t_2t_{24}
+\overline{q}^{\frac{1}{2}}t_1t_{14}+q^{\frac{3}{2}}\alpha t_4,
\end{align*}
then (\ref{eq:t123t0}) follows by using (\ref{eq:t123t134}) and $t_2t_1=q^{\frac{1}{2}}t'_{12}+\overline{q}^{\frac{1}{2}}t_{12}$.

By Figure \ref{fig:t0-times-t0},
\begin{align*}
z_2t_0&=q(q^{\frac{1}{2}}t_4t_{23}-qt_{234})t_{234}-q^2t_4^2-t_{23}^2+q\alpha+\overline{q}(\overline{q}^{\frac{1}{2}}t_1t_{23}-t_{123})t_{123}
-t_{23}^2  \\
&\ \ \ \ -\overline{q}^2t_1^2+\overline{q}\alpha+t_{23}^2+t_{14}(t_{123}t_{234}-qt_1t_4-\overline{q}t_{23}t_0-t_{14}).
\end{align*}
Then (\ref{eq:t0t0}) follows by substituting (\ref{eq:z2}).
\end{proof}

\subsection{Proof of Theorem \ref{thm:presentation-1} and Theorem \ref{thm:basis-1}}

Given a 3-manifold $M$, let $\mathcal{X}(\pi_1(M))$ denote the ${\rm SL}(2,\mathbb{C})$-character variety of $\pi_1(M)$, and define
$$\epsilon:\mathcal{S}(M)\otimes_{q^{1/2}=-1}\mathbb{C}\to\mathbb{C}[\mathcal{X}(\pi_1(M))]$$
by sending a link $\mathbf{l}=\bigsqcup_{i=1}^m\mathbf{k}_i$ (the $\mathbf{k}_i$ being components) to the function
$$\epsilon(\mathbf{l}):\mathcal{X}(\pi_1(M))\to\mathbb{C}, \qquad  \chi\to{\prod}_{i=1}^m(-\chi([\mathbf{k}_i])),$$
where $[\mathbf{k}_i]$ denotes the conjugacy class determined by $\mathbf{k}_i$.

The following is \cite[Lemma 4.1]{Ch25}:
\begin{lem}\label{lem:strategy}
Suppose $\mathcal{S}(M)$ is a torsion-free $R$-module. If $\mathcal{S}(M)=R\langle\mathcal{B}\rangle$ and $\epsilon(\mathcal{B})$ is $\mathbb{C}$-linearly independent, then $\mathcal{B}$ is a basis for $\mathcal{S}(M)$.
\end{lem}

Recall
$$\mathcal{C}=\big\{t_1^{i_1}t_2^{i_2}t_3^{i_3}t_4^{i_4}at_{12}^{j_1}t_{23}^{j_2}t_{34}^{j_3}t_{14}^{j_4}
\colon i_1,\ldots,i_4,j_1,\ldots,j_4\ge 0,\ a\in\mathcal{A}\big\},$$
where
\begin{align*}
\mathcal{A}=\ &\{1,t_0,t_{123},t_{124},t_{134},t_{234}\}  \\
&\cup\{t_{13}^k,\  t_{13}^kt_0, \ t_{13}^kt_{123},\ t_{13}^kt_{134}, \ t_{24}^k, \ t_{24}^kt_0, \ t_{24}^kt_{124}, \ t_{24}^kt_{234} \colon k\ge 1\}.
\end{align*}

For $1\le k_1,\ldots,k_r\le 4$, let $\mathsf{t}_{k_1\cdots k_r}\in\mathbb{C}[\mathcal{X}(F_4)]$ denote the function sending a character $\chi$ to $\chi(x_{k_1}\cdots x_{k_r})$.
By \cite[Theorem 3.4]{Ch25},
$$\mathsf{B}=\big\{\mathsf{t}_1^{i_1}\mathsf{t}_2^{i_2}\mathsf{t}_3^{i_3}\mathsf{t}_4^{i_4}\mathsf{t}_{12}^{j_1}\mathsf{t}_{23}^{j_2}
\mathsf{t}_{34}^{j_3}\mathsf{t}_{14}^{j_4}\mathsf{a}\colon i_1,\ldots,i_4,j_1,\ldots,j_4\ge 0, \ \mathsf{a}\in\mathsf{A}\big\}$$
is a basis for $\mathbb{C}[\mathcal{X}(F_4)]$ as a vector space over $\mathbb{C}$, where
\begin{align*}
\mathsf{A}&=\{1,\mathsf{t}_{13},\mathsf{t}_{24},\mathsf{t}_{13}\mathsf{t}_{24}\}
\cup\big\{\mathsf{t}_{13}^k,\mathsf{t}_{13}^k\mathsf{t}_{24},\mathsf{t}_{24}^k,\mathsf{t}_{24}^k\mathsf{t}_{13}\colon k\ge 2\big\}  \\
&\ \ \ \ \cup\big\{\mathsf{t}_{13}^k\mathsf{t}_{123},\mathsf{t}_{13}^k\mathsf{t}_{134},\mathsf{t}_{24}^k\mathsf{t}_{124},
\mathsf{t}_{24}^k\mathsf{t}_{234}\colon k\ge 0\big\}.
\end{align*}

For $M=\Sigma\times[0,1]$, we have $\pi_1(M)\cong F_4$, which is generated by the classes of $t_1,\ldots,t_4$.
By definition, $\epsilon(t_{k_1\cdots k_r})=-\mathsf{t}_{k_1\cdots k_r}$.
In view of (\ref{eq:t13t24}), it is not difficult to see that $\epsilon(\mathcal{C})$ is related to $\mathsf{B}$ via a linear automorphism,
so $\epsilon(\mathcal{C})$ is $\mathbb{C}$-linearly independent.
By \cite[Theorem 2.3 (b)]{Pr99}, $\mathcal{S}(\Sigma)$ is torsion-free.
By Lemma \ref{lem:strategy}, the proof of Theorem \ref{thm:basis-1} reduces to showing $\mathcal{S}(\Sigma)=R\langle\mathcal{C}\rangle$.

Set $|t_{k_1\cdots k_r}|=r$; in particular, $|t_0|=4$.  
Call $t_1,t_2,t_3,t_4$ {\it special}, and call the other generators {\it ordinary}. Among the ordinary generators, call $t_{12}$, $t_{23}$, $t_{34}$, $t_{14}$ {\it light}, and call $t_0$, $t_{13}$, $t_{24}$, $t_{123}$, $t_{124}$, $t_{134}$, $t_{234}$ {\it heavy}.

Given a monomial $u=g_1\cdots g_m$ with $g_i\in\mathcal{G}$, define its {\it degree} by $|u|=\sum_{i=1}^m|g_i|$, set
$|u|^\circ=\#\{i\colon |g_i|=1\}$, and put
$$\|u\|=(|u|,|u|^\circ,|\ddot{u}|),   \qquad   \|u\|^\ast=(|u|,|\ddot{u}|),$$
where $\ddot{u}$ is the product of the $g_i$'s with $g_i$ heavy.

Given monomials $u,v$, write $\|v\|\prec \|u\|$ if one of the following conditions holds: (i) $|v|<|u|$;
(ii) $|v|=|u|$ and $|v|^\circ<|u|^\circ$; (iii) $|v|=|u|$, $|v|^\circ=|u|^\circ$ and $|\ddot{v}|<|\ddot{u}|$.
For instance, $\|t_{12}t_{34}\|\prec\|t_{13}t_{24}\|\prec\|t_1t_{234}\|$, and $\|t_{23}t_0\|\prec\|t_{123}t_{234}\|$.
Write $\|v\|^\ast\prec \|u\|^\ast$ if one of the following conditions holds: (i) $|v|<|u|$; (ii) $|v|=|u|$ and $|\ddot{v}|<|\ddot{u}|$.

\begin{lem}\label{lem:commutator}
For any $g,g'\in\mathcal{G}$, either $gg'=g'g$, or $gg'-q^{\varepsilon}g'g=\sum_j\beta_jv_j$, for some $\varepsilon\in\mathbb{Z}$, $\beta_j\in R$ and monomials $v_j$ with $\|v_j\|\prec\|gg'\|$. In the latter case, if $g,g'$ are both ordinary, then also $\|v_j\|^\ast\prec\|gg'\|^\ast$ for all $j$.
\end{lem}

\begin{proof}
This is a coarser reformulation of the commuting relations and commutator relations, which have been verified in Section \ref{sec:verify}.
\end{proof}

Call a monomial {\it normal} if it has the form $t_1^{i_1}\cdots t_4^{i_4}wt_{12}^{j_1}t_{23}^{j_2}t_{34}^{j_3}t_{14}^{j_4}$, where $w$ is a product of heavy generators.

Repeatedly applying Lemma \ref{lem:commutator}, we can {\it normalize} each monomial $u$, which means rewriting it as $\gamma N(u)+\sum_i\beta_iv_i$,
for some $\gamma\in R^\times$, $\beta_i\in R$ and normal monomials $N(u)$ and $v_i$ with $\epsilon(N(u))=\epsilon(u)$ and $\|v_i\|\prec\|u\|$.

For $\beta\in R$ and monomials $u,u'$, write $u\equiv \beta u'$ and say that $u$ can be {\it displaced} by $\beta u'$, if $\|u\|^\ast=\|u'\|^\ast$ and $u-\beta u'=\sum_i\eta_iv_i$ for some $\eta_i\in R$ and monomials $v_i$ with $\|v_i\|^\ast\prec\|u\|^\ast$. Write $u\sim u'$ if $u\equiv\beta u'$ for some $\beta\in R^\times$.
Write $u\sim 0$ and say that $u$ can be {\it reduced}, if $u=\sum_i\eta_iv_i$ for some $\eta_i\in R$ and monomials $v_i$ with $\|v_i\|^\ast\prec\|u\|^\ast$.

\begin{rmk}
\rm Think of a normal monomial as composed of three layers, namely, the special-, heavy- and light layer; each one is allowed to be empty. We shall always perform reductions in the leftmost part of the heavy layer, so that any $t_i$ produced in the reductions can be incorporated into the special layer. Then we immediately normalize each term whenever necessary.
\end{rmk}

\begin{lem}\label{lem:reduce}
Each of $t_0^2$, $t_0t_{123},t_0t_{124},t_0t_{134},t_0t_{234}$, $t_{123}t_{24}$, $t_{123}^2$, $t_{123}t_{124}$, $t_{123}t_{234}$ can be reduced;
$t_{13}t_{24}\equiv t_{24}t_{13}\equiv 2t_0$, and $t_{123}t_{134}\sim t_{13}t_0\sim t_0t_{13}$.
\end{lem}

\begin{proof}
By (\ref{eq:t0t0}), $t_0^2\sim 0$.

By (\ref{eq:t123t0}), (\ref{eq:t123-vs-t0}), $t_0t_{123}\sim 0$; by symmetry, $t_0g\sim 0$ for $g\in\{t_{124},t_{134},t_{234}\}$.

By (\ref{eq:t13t234}), (\ref{eq:t13-vs-t234}), $t_{234}t_{13}\sim 0$; by symmetry, $t_{123}t_{24}\sim 0$.

By (\ref{eq:t123t123}), $t_{123}^2\sim 0$; by (\ref{eq:t123t124}), $t_{123}t_{124}\sim 0$; by (\ref{eq:t123t234}), $t_{123}t_{234}\sim 0$.

By (\ref{eq:t13t24}), $t_{13}t_{24}\equiv 2t_0$; by symmetry, $t_{24}t_{13}\equiv 2t_0$.

By (\ref{eq:t123t134}), $t_{123}t_{134}\sim t_{13}t_0$; by (\ref{eq:t13-vs-t0}), $t_{13}t_0\sim t_0t_{13}$.
\end{proof}

\begin{lem}
Each normal monomial belongs to $R\langle\mathcal{C}\rangle$.
\end{lem}

\begin{proof}
We prove $u\in R\langle\mathcal{C}\rangle$ for normal monomials $u$ by induction on $\|u\|^\ast$.
Clearly, $u\in R\langle\mathcal{C}\rangle$ if $|u|\le 3$.
Suppose $|u|>3$ and that $u'\in R\langle\mathcal{C}\rangle$ for all normal monomials $u'$ with $\|u'\|^\ast\prec\|u\|^\ast$.

Let $u=t_1^{i_1}\cdots t_4^{i_4}g_1\cdots g_mt_{12}^{j_1}t_{23}^{j_2}t_{34}^{j_4}t_{14}^{j_4}$, with each $g_i$ heavy. We shall apply Lemma \ref{lem:reduce} repeatedly. By Lemma \ref{lem:commutator}, modulo smaller terms, we may freely interchange any adjacent $g_i$'s. 
\begin{enumerate}
  \item Whenever $g_i=g_j=t_0$ for some $i<j$, we can move $g_i,g_j$ to the leftmost positions, and reduce $t_0^2$, so as to reduce $u$.
  \item Suppose $g_i=t_0$ for exactly one $i$. We may assume $g_1=t_0$.
  \begin{enumerate}
    \item When $g_j\in\{t_{123},t_{124},t_{134},t_{234}\}$ for some $j$, we can move $g_j$ to meet $t_0$, and then reduce $t_0g_j$,
          so as to reduce $u$.
    \item Otherwise, $g_j\in\{t_{13},t_{24}\}$ for all $j>1$.
          If $u=t_0t_{13}^{m-1}$ or $u=t_0t_{24}^{m-1}$, we are done;
          otherwise, if $\{g_j,g_{j+1}\}=\{t_{13},t_{24}\}$ for some $j$, we can move $g_jg_{j+1}$ to the leftmost position, 
          displace it by $2t_0$, and then reduce $t_0^2$ as in Case 1.
  \end{enumerate}
  \item Suppose $\max\{|g_1|,\ldots,|g_m|\}=3$. Let $i_0=\min\{i\colon |g_i|=3\}$. We may assume $g_{i_0}=t_{123}$;
        the cases $g_{i_0}\in\{t_{124},t_{134},t_{234}\}$ are similar.
  \begin{enumerate}
    \item If $g_j=t_{13}$ for all $j\ne i_0$, then $u\sim t_{13}^{m-1}t_{123}$, and we are done.
    \item Otherwise, $g_j\ne t_{13}$ for some $j$, i.e. $g_j\in\{t_{24},t_{123},t_{124},t_{134},t_{234}\}$. Moving $g_{i_0}$, $g_j$ to the leftmost positions if necessary, we may assume $i_0=1$ and $j=2$.

          If $g_2=t_{134}$, then we can displace $g_1g_2$ by $qt_{13}t_0$, and turn to Case 2;
          if $g_2\in\{t_{24},t_{123},t_{124},t_{234}\}$, then we can reduce $g_1g_2$.
  \end{enumerate}
  \item Suppose all $g_i\in\{t_{13},t_{24}\}$. If $\{g_j,g_{j+1}\}=\{t_{13},t_{24}\}$ for some $j$, then we can move $g_jg_{j+1}$ to the leftmost position, displace it by $2t_0$, and turn to Case 2.
        Otherwise, $u=t_{13}^{m}$ or $u=t_{24}^{m}$.
\end{enumerate}

This completes the proof.
\end{proof}

Therefore, Theorem \ref{thm:basis-1} is true. Also established is Theorem \ref{thm:presentation-1}, since it has been shown that each monomial can be reduced to $R\langle\mathcal{C}\rangle$ via the relations given in Theorem \ref{thm:presentation-1}.


\section{The skein algebra of $\Sigma_{2,0}$}

We may construct $\Sigma_{2,0}\times[0,1]$ by attaching a $2$-handle $H$ to $\Sigma\times[0,1]$ along the circle $\partial\Sigma\times\{1/2\}$.
The inclusion $\Sigma\times[0,1]\hookrightarrow\Sigma_{2,0}\times[0,1]$ induces an epimorphism
$\mathcal{S}(\Sigma_{2,0})\twoheadrightarrow\mathcal{S}(\Sigma)$.

\begin{figure}[h]
  \centering
  \includegraphics[width=13cm]{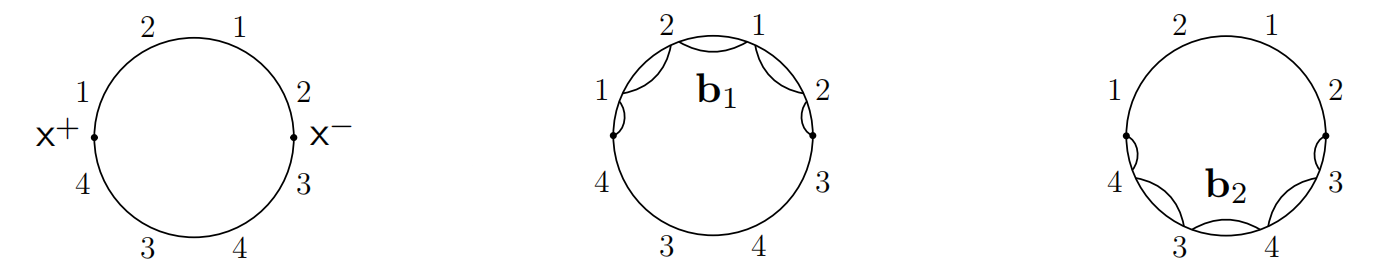}\\
  \caption{}\label{fig:attaching}
\end{figure}

Fix two points $\mathsf{x}^{\pm}\in\partial\Sigma$ and take two arcs $\mathbf{b}_1,\mathbf{b}_2$ with $\partial\mathbf{b}_1=\partial\mathbf{b}_2=\{\mathsf{x}^{\pm}\}$ that stay close to $\partial\Sigma$, as illustrated in Figure \ref{fig:attaching}. Note that $\mathbf{b}_1\cup\mathbf{b}_2$ is isotopic to $\partial\Sigma$.

Identify $\Sigma$ with $\Sigma\times\{1/2\}$.
Let $L$ denote the set of $\partial$-relative isotopy classes of 1-submanifolds $\mathbf{u}\subset\Sigma\times[0,1]$ with $\partial\mathbf{u}=\{\mathsf{x}^{\pm}\}$ and such that its projection to $\Sigma$ is perpendicular to $\partial\Sigma$ at $\mathsf{x}^{\pm}$ and disjoint from $\mathbf{b}_1,\mathbf{b}_2$.
For each $\mathbf{u}\in L$, let ${\rm tr}(\mathbf{u}\mathbf{b}_i)$ denote the link $\mathbf{u}\cup\mathbf{b}_i$, and put
\begin{align*}
{\rm Sl}(\mathbf{u})={\rm tr}(\mathbf{u}\mathbf{b}_1)-{\rm tr}(\mathbf{u}\mathbf{b}_2)\in\mathcal{S}(\Sigma).
\end{align*}

We have the following simple but useful observation:
\begin{lem}\label{lem:sliding}
If $\mathbf{l}_1,\mathbf{l}_2\subset\Sigma\times[0,1]$ are links such that $\mathbf{l}_2$ is obtained by sliding some component of $\mathbf{l}_1$ along the attaching circle of $H$, then there exists $\mathbf{u}\in L$ such that $\mathbf{l}_1-\mathbf{l}_2={\rm Sl}(\mathbf{u})$ or
$\mathbf{l}_2-\mathbf{l}_1={\rm Sl}(\mathbf{u})$ in $\mathcal{S}(\Sigma)$.
\end{lem}

\begin{proof}
Using an appropriate isotopy we may modify one of $\mathbf{l}_1,\mathbf{l}_2$ into the form ${\rm tr}(\mathbf{u}\mathbf{b}_1)$ for some $\mathbf{u}\in L$, then we can isotope the other into ${\rm tr}(\mathbf{u}\mathbf{b}_2)$.
\end{proof}

\begin{figure}[h]
  \centering
  \includegraphics[width=11.7cm]{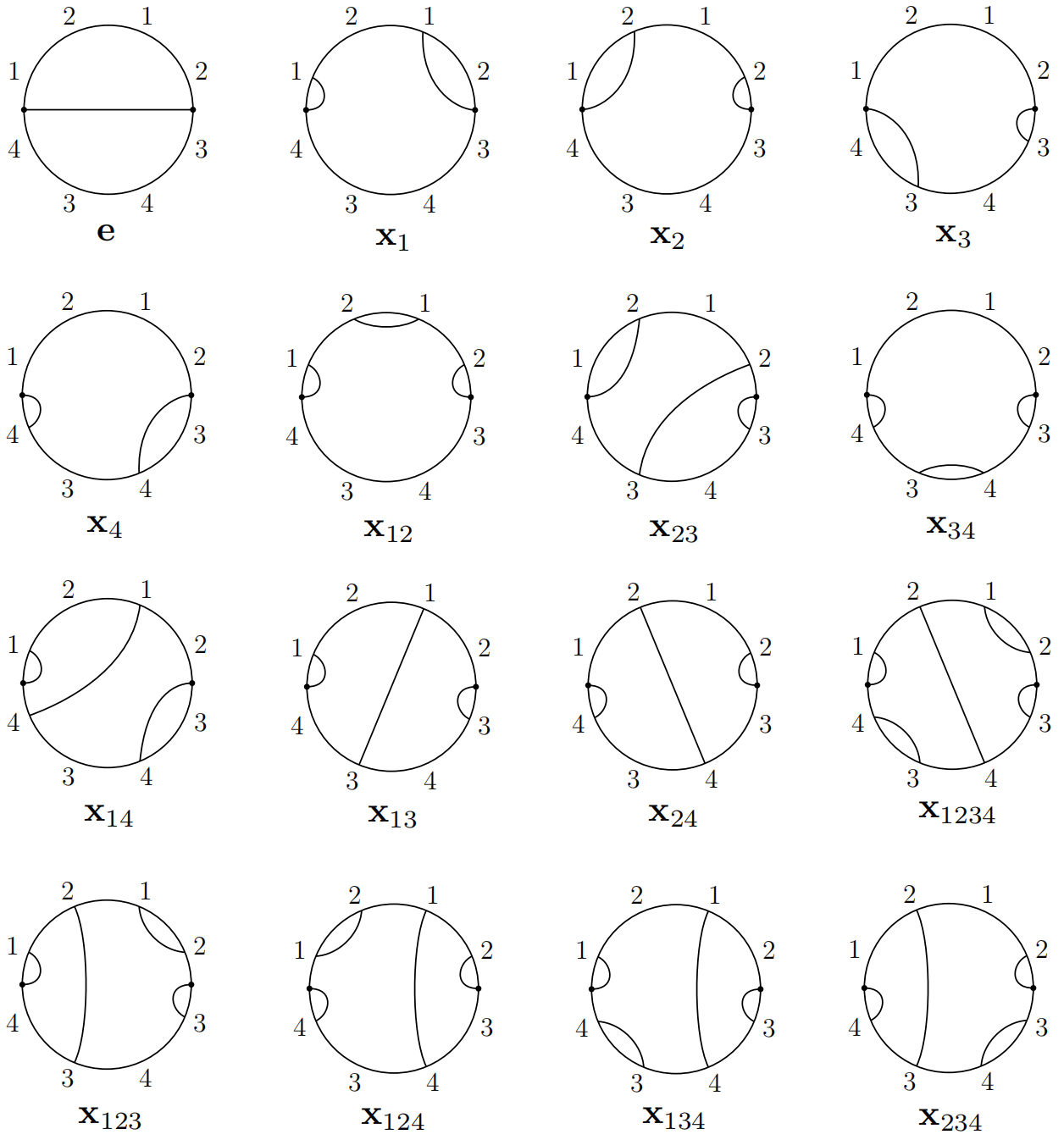}\\
  \caption{}\label{fig:arc}
\end{figure}

Let $X$ denote the subset of $L$ consisting of unshortenable simple arcs.
Define an equivalence relation $\sim$ on $X$ by declaring $\mathbf{u}\sim\mathbf{u}'$ if $\mathbf{u}'=\beta\mathbf{u}+\sum_ia_i\mathbf{v}_i$ in $\mathcal{E}_{\mathsf{x}^+,\mathsf{x}^-}$ for some $a_i\in\mathcal{S}(\Sigma)$,
$\mathbf{v}_i\in X$ with $|\mathbf{v}_i|<|\mathbf{u}|$ and $\beta\in R^\times$.

\begin{lem}\label{lem:equivalence-arc}
Let $X_0$ denote the set of arcs in Figure \ref{fig:arc}. Each element of $X$ is equivalent to one in $X_0$.
\end{lem}

\begin{proof}
Decompose $X$ as $X=\bigsqcup_{n=0}^4X(n)$, with $X(n)=\{\mathbf{a}\in X\colon |\mathbf{a}|=n\}$.

It is easy to verify that for each $\mathbf{a}\in \bigsqcup_{n=0}^2X(n)$, there exists $\mathbf{b}\in X_0$ with
${\rm supp}(\mathbf{b})={\rm supp}(\mathbf{a})$ and $\mathbf{a}\sim\mathbf{b}$.

When $\mathbf{a}\in X(3)$, up to symmetry we may assume ${\rm supp}(\mathbf{a})=\{1,2,3\}$. As illustrated in Figure \ref{fig:equivalence-2}, there are $7$ such elements, which are equivalent to each other, in particular, to $\mathbf{x}_{123}\in X_0$.

Note that $X(4)=X^u(4)\sqcup X^d(4)$, where $X^u(4)$ (resp. $X^d(4)$) consists of arcs through which $\mathsf{x}^+$ is connected to $1$ (resp. to $4$). As is shown in Figure \ref{fig:equivalence-3}, elements of $X^u(4)$ are equivalent to each other. By symmetry, elements of $X^d(4)$ are equivalent to each other, too. Furthermore, an element of $X^u(4)$ is equivalent to one of $X^d(4)$. Thus, all elements of $X(4)$ are equivalent to each other, in particular, to $\mathbf{x}_{1234}\in X_0$.
\end{proof}

\newpage
\begin{figure}[H]
  \centering
  \includegraphics[width=12.2cm]{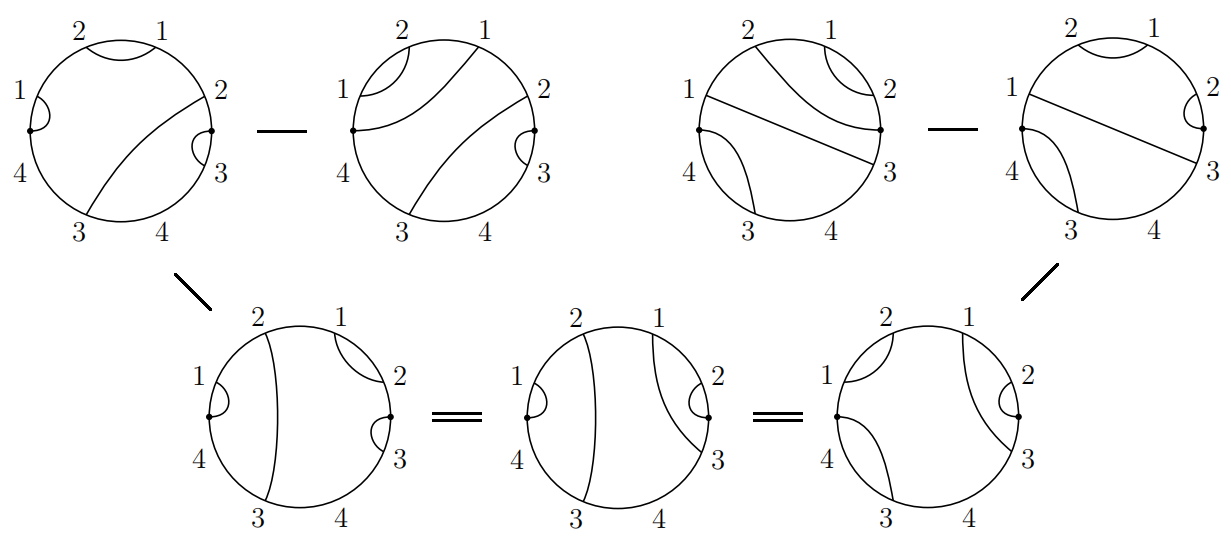}\\
  \caption{By connecting $\mathbf{a},\mathbf{a}'\in X(3)$ via a double line, we mean that $\mathbf{a}'=\beta\mathbf{a}+\gamma\mathbf{s}\mathbf{c}$ for some $\beta\in R^\times$, $\gamma\in R$, $\mathbf{c}\in X$ with $|\mathbf{c}|=1$, and a loop $\mathbf{s}$ with $|\mathbf{s}|=2$.}\label{fig:equivalence-2}
\end{figure}

\begin{figure}[H]
  \centering
  \includegraphics[width=9cm]{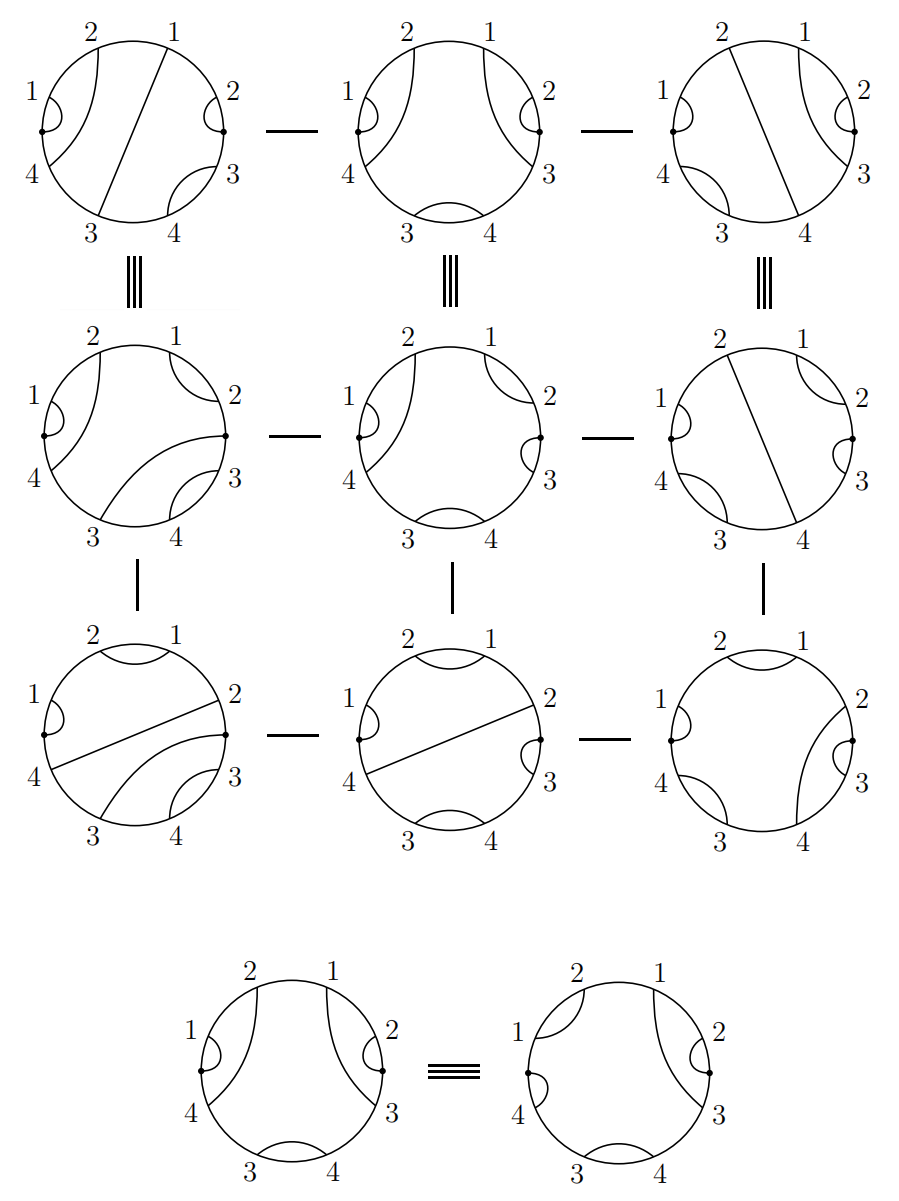}\\
  \caption{The meaning of a triple line is similar to that above.}\label{fig:equivalence-3}
\end{figure}

\newpage

\begin{lem}\label{lem:kernel}
The kernel of the morphism $\mathcal{S}(\Sigma)\twoheadrightarrow \mathcal{S}(\Sigma_{2,0})$
coincides with the left ideal generated by ${\rm Sl}(\mathbf{a})$ for $\mathbf{a}\in X_0$.
\end{lem}

\begin{proof}
By \cite[Proposition 2.2 (5)]{Pr99}, $\mathcal{S}(\Sigma_{2,0})\cong\mathcal{S}(\Sigma)/\mathcal{Q}$ as $R$-modules,
where $\mathcal{Q}$ is the $R$-submodule of $\mathcal{S}(\Sigma)$ generated by $\mathbf{l}_1-\mathbf{l}_2$ for all links $\mathbf{l}_1,\mathbf{l}_2$ such that $\mathbf{l}_2$ is obtained by sliding some component of $\mathbf{l}_1$ along the attaching circle of $H$.
By Lemma \ref{lem:sliding}, $\mathcal{Q}=R\{{\rm Sl}(\mathbf{u})\colon \mathbf{u}\in L\}$.

Let $L_0$ denote the subset of $L$ consisting of simple arcs. By resolving crossings, we can write each element of $L$ in the form $\sum_ia_i\mathbf{c}_i$, with $a_i\in\mathcal{S}(\Sigma)$, $\mathbf{c}_i\in L_0$.
By Lemma \ref{lem:chop-up}, $L_0\subset\mathcal{S}(\Sigma)\langle X\rangle$. By Lemma \ref{lem:equivalence-arc},
$\mathcal{S}(\Sigma)\langle X\rangle=\mathcal{S}(\Sigma)\langle X_0\rangle$. Thus, $L\subset \mathcal{S}(\Sigma)\langle X_0\rangle$, implying
$\mathcal{Q}=\mathcal{S}(\Sigma)\langle\{{\rm Sl}(\mathbf{a})\colon\mathbf{a}\in X_0\}\rangle$.
\end{proof}


\begin{proof}[Proof of Theorem \ref{thm:presentation-2}]
By Lemma \ref{lem:kernel}, it suffices to show that the relations ${\rm Sl}(\mathbf{a})=0$, $\mathbf{a}\in X_0$ give rise to (\ref{eq:Sl-0})--(\ref{eq:Sl-234}).

As shown in Figure \ref{fig:relation-e}, ${\rm tr}(\mathbf{e}\mathbf{b}_1)=t_{12}t'_{12}-qt_2^2-\overline{q}t_1^2+\alpha$. 
By symmetry, ${\rm tr}(\mathbf{e}\mathbf{b}_2)=t_{34}t'_{34}-qt_4^2-\overline{q}t_3^2+\alpha$.
Hence
${\rm Sl}(\mathbf{e})=0$ is equivalent to (\ref{eq:Sl-0}).

By Figure \ref{fig:relation-x1}, ${\rm Sl}(\mathbf{x}_1)=0$ reads (\ref{eq:Sl-1}). 
By symmetry, ${\rm Sl}(\mathbf{x}_2)=0$, ${\rm Sl}(\mathbf{x}_3)=0$, ${\rm Sl}(\mathbf{x}_4)=0$ respectively read (\ref{eq:Sl-2})--(\ref{eq:Sl-4}).

By Figure \ref{fig:relation-x12}, ${\rm Sl}(\mathbf{x}_{12})=0$ reads (\ref{eq:Sl-12}). 
By symmetry, ${\rm Sl}(\mathbf{x}_{34})=0$ reads (\ref{eq:Sl-34}).

By Figure \ref{fig:relation-x13}, ${\rm Sl}(\mathbf{x}_{13})=0$ reads $y_{123}=y_{134}$, which is equivalent to (\ref{eq:Sl-13}). 
By symmetry, ${\rm Sl}(\mathbf{x}_{24})=0$ is equivalent to (\ref{eq:Sl-24}).

By Figure \ref{fig:relation-x14}, ${\rm Sl}(\mathbf{x}_{14})=0$ reads
\begin{align*}
0&=\overline{q}^{\frac{1}{2}}t_{34}y_{134}-\overline{q}(\overline{q}(t_3^2-1)t_{14}-\overline{q}^{\frac{3}{2}}t_3t_{134})-y_{124}  \\
&=\overline{q}^{\frac{1}{2}}t_{34}(\overline{q}^{\frac{1}{2}}t'_{34}t_{14}-\overline{q}^{\frac{3}{2}}t_4t_{134}+\overline{q}^2t_{13})
+\overline{q}^2(1-t_3^2)t_{14}+\overline{q}^{\frac{5}{2}}t_3t_{134} \\
&\ \ \ \ -(\overline{q}^{\frac{1}{2}}t'_{12}t_{24}-\overline{q}^{\frac{3}{2}}t_2t_{124}+\overline{q}^2t_{14})  \\
&=\overline{q}t_{34}t'_{34}t_{14}-\overline{q}^2(qt_4t_{34}+(\overline{q}^{\frac{1}{2}}-q^{\frac{3}{2}})t_3)t_{134}
+\overline{q}^{\frac{5}{2}}t_{34}t_{13}+\overline{q}^2(1-t_3^2)t_{14}  \\
&\ \ \ \ +\overline{q}^{\frac{5}{2}}t_3t_{134}-(\overline{q}^{\frac{1}{2}}t'_{12}t_{24}-\overline{q}^{\frac{3}{2}}t_2t_{124}+\overline{q}^2t_{14})  \\
&=(\overline{q}t_{34}t'_{34}-\overline{q}^2t_3^2)t_{14}-(\overline{q}t_4t_{34}-\overline{q}^{\frac{1}{2}}t_3)t_{134}
+\overline{q}^{\frac{5}{2}}t_{34}t_{13}-\overline{q}^{\frac{1}{2}}t'_{12}t_{24}+\overline{q}^{\frac{3}{2}}t_2t_{124},
\end{align*}
which, due to (\ref{eq:Sl-1}), is equivalent to (\ref{eq:Sl-14}).
By symmetry, ${\rm Sl}(\mathbf{x}_{23})=0$ is equivalent to (\ref{eq:Sl-23}).

By Figure \ref{fig:relation-x123}, ${\rm Sl}(\mathbf{x}_{123})=0$ is equivalent to
\begin{align*}
t_4t_0&=qt'_{34}t_{124}-qt_2y_{134}+t_1t_{23}   \\
&=qt'_{34}t_{124}-q^{\frac{1}{2}}t_2t'_{34}t_{14}+\overline{q}^{\frac{1}{2}}t_2t_4t_{134}-\overline{q}t_2t_{13}+t_1t_{23},
\end{align*}
which is just (\ref{eq:Sl-123}).
By symmetry, ${\rm Sl}(\mathbf{x}_{124})=0$, ${\rm Sl}(\mathbf{x}_{134})=0$, ${\rm Sl}(\mathbf{x}_{234})=0$ are respectively equivalent to (\ref{eq:Sl-124}), (\ref{eq:Sl-134}), (\ref{eq:Sl-234}).

By Figure \ref{fig:relation-x1234}, ${\rm Sl}(\mathbf{x}_{1234})=0$ reads $z_1=z_2$, which is equivalent to (\ref{eq:Sl-1234}).
\end{proof}

\newpage

\section{Appendix: pictorial formulas}

\begin{conv}
\rm 
In each figure, when a term is underlined, it is expanded in detail in another formula immediately below.
\end{conv}

\begin{figure}[H]
  \centering
  \includegraphics[width=12.5cm]{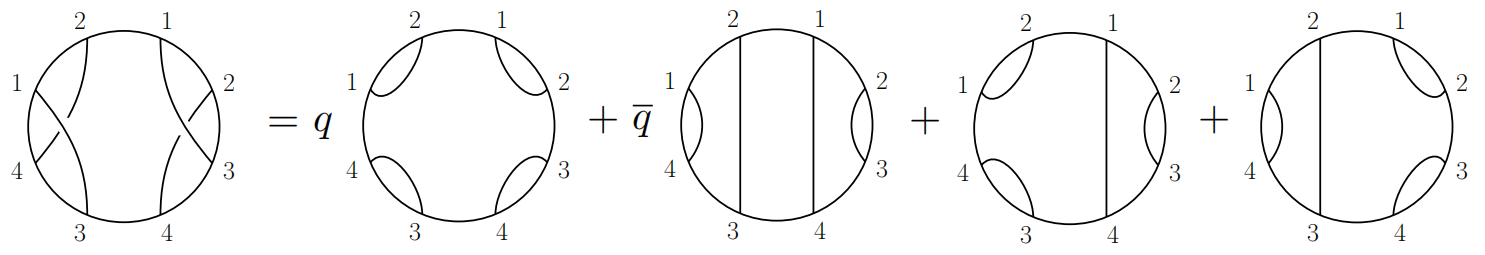}\\
  \caption{Deducing the formula for $t_{13}t_{24}$.}\label{fig:t13-times-t24}
\end{figure}

\begin{figure}[H]
  \centering
  \includegraphics[width=12.5cm]{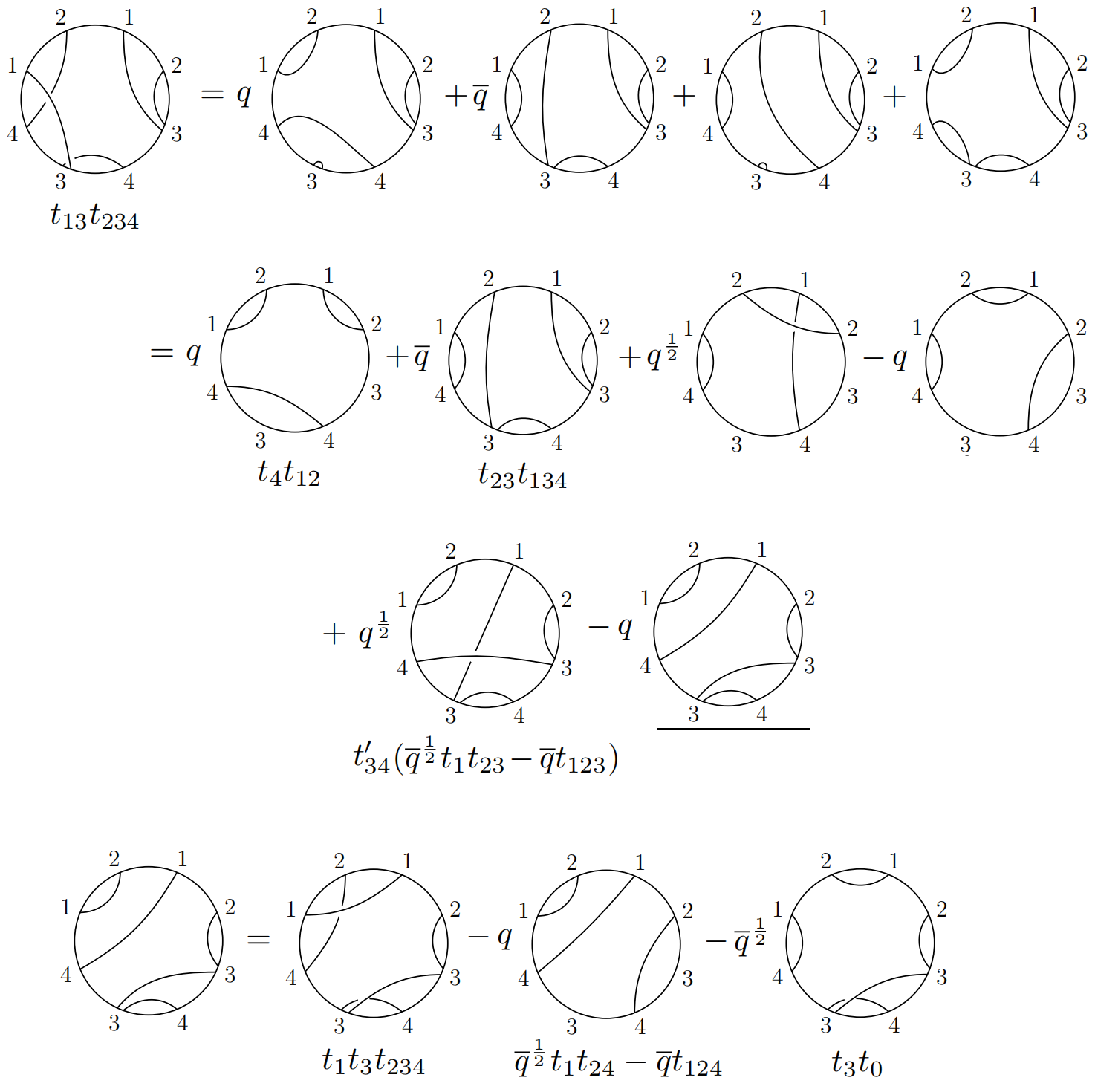}\\
  \caption{Deducing the formula for $t_{13}t_{234}$.}\label{fig:t13-times-t234}
\end{figure}

\begin{figure}[H]
  \centering
  \includegraphics[width=12.9cm]{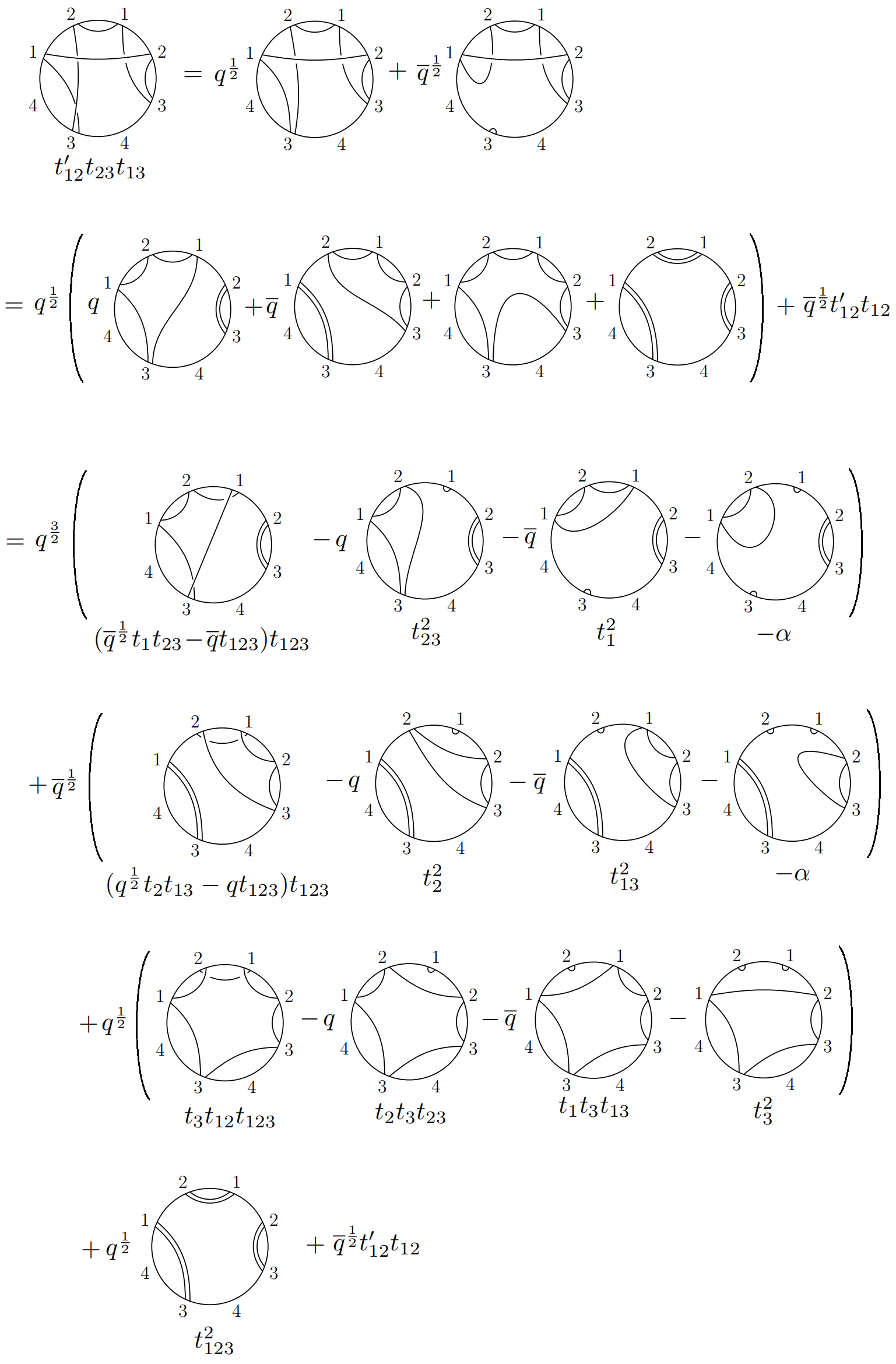}\\
  \caption{Deducing the formula for $t_{123}^2$.}\label{fig:t123-times-t123}
\end{figure}

\begin{figure}[H]
  \centering
  \includegraphics[width=12.5cm]{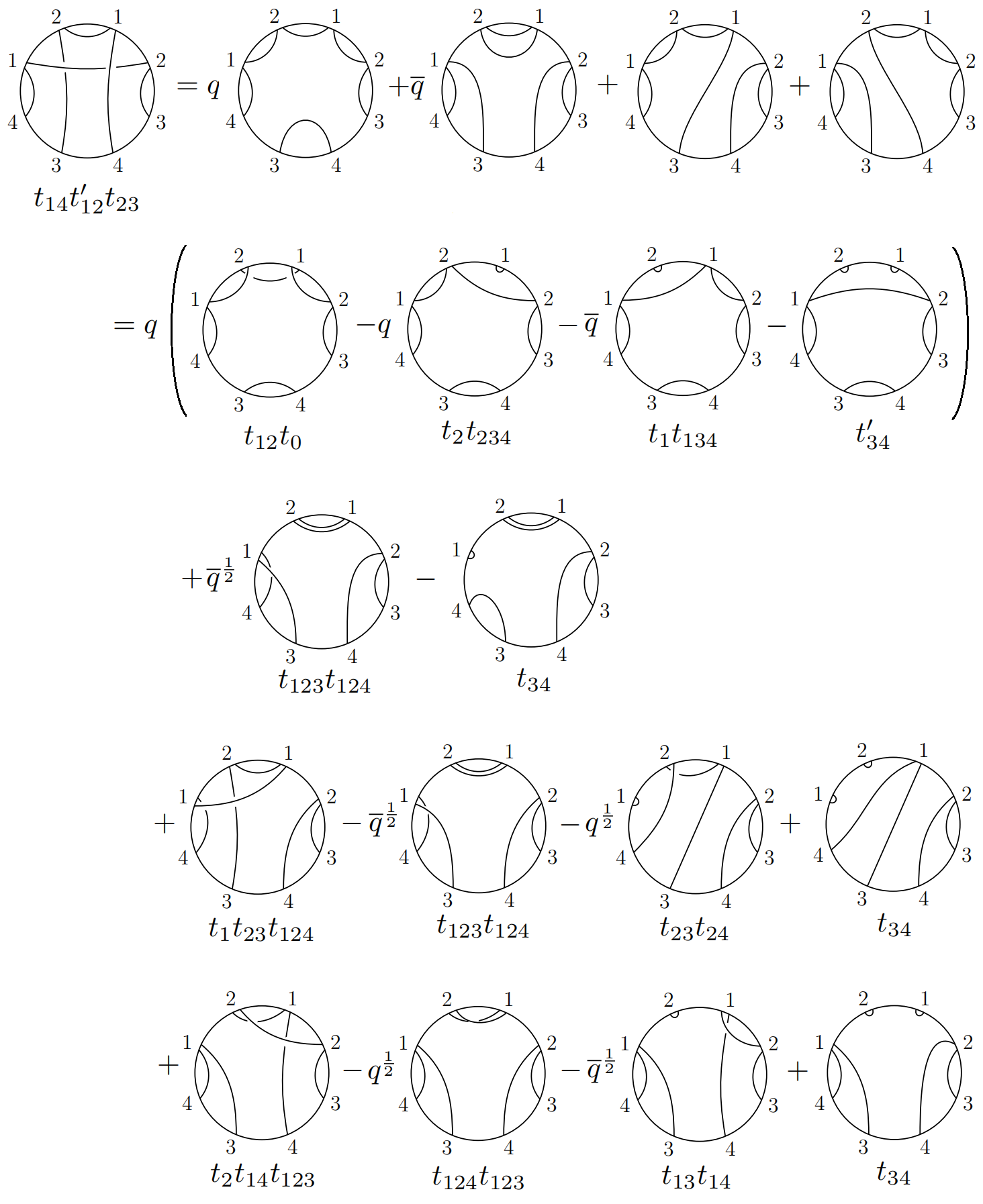}\\
  \caption{Deducing the formula for $t_{123}t_{124}$.}\label{fig:t123-times-t124}
\end{figure}

\begin{figure}[H]
  \centering
  \includegraphics[width=12.5cm]{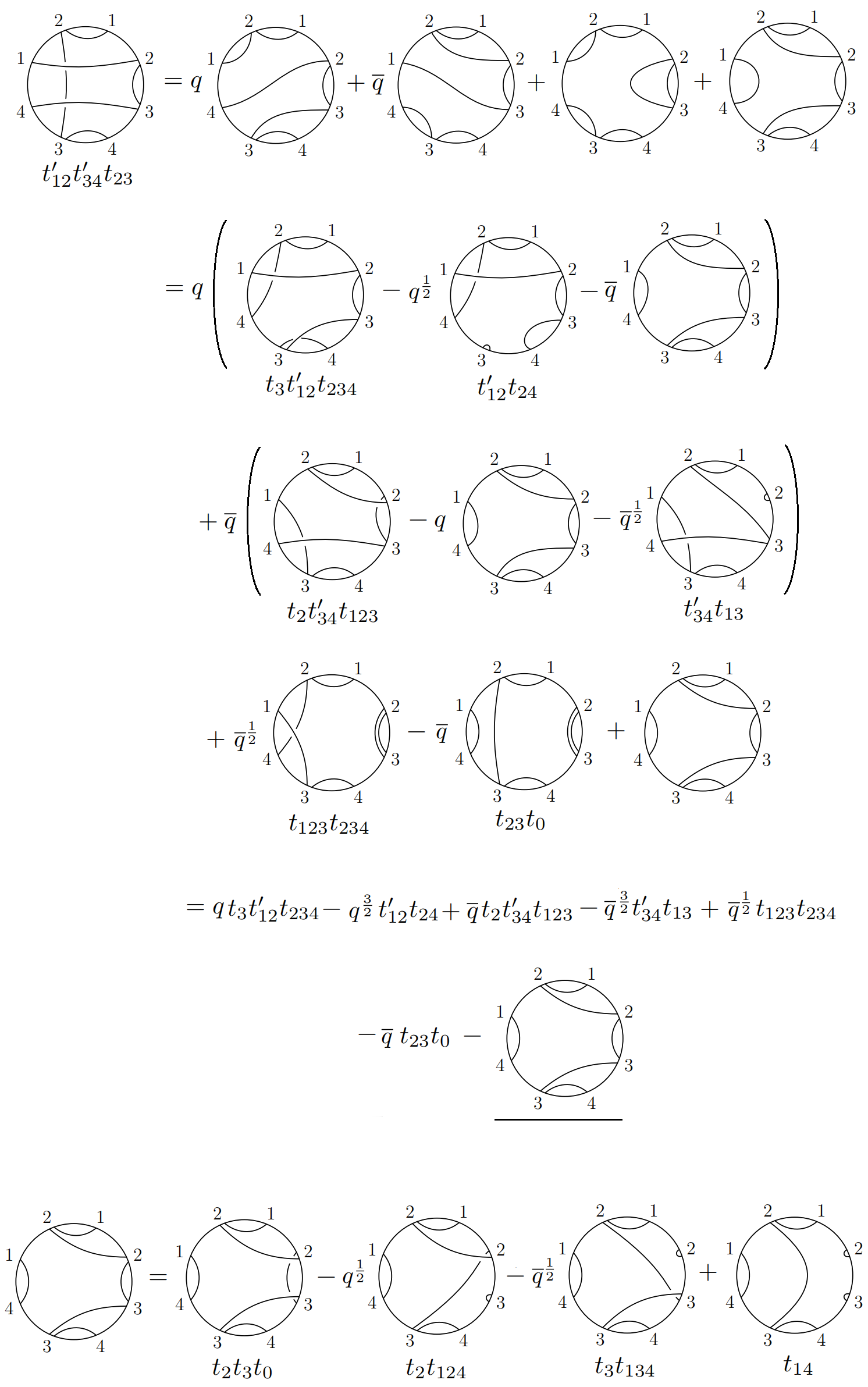}\\
  \caption{Deducing the formula for $t_{123}t_{234}$.}\label{fig:t123-times-t234}
\end{figure}

\begin{figure}[H]
  \centering
  \includegraphics[width=12.3cm]{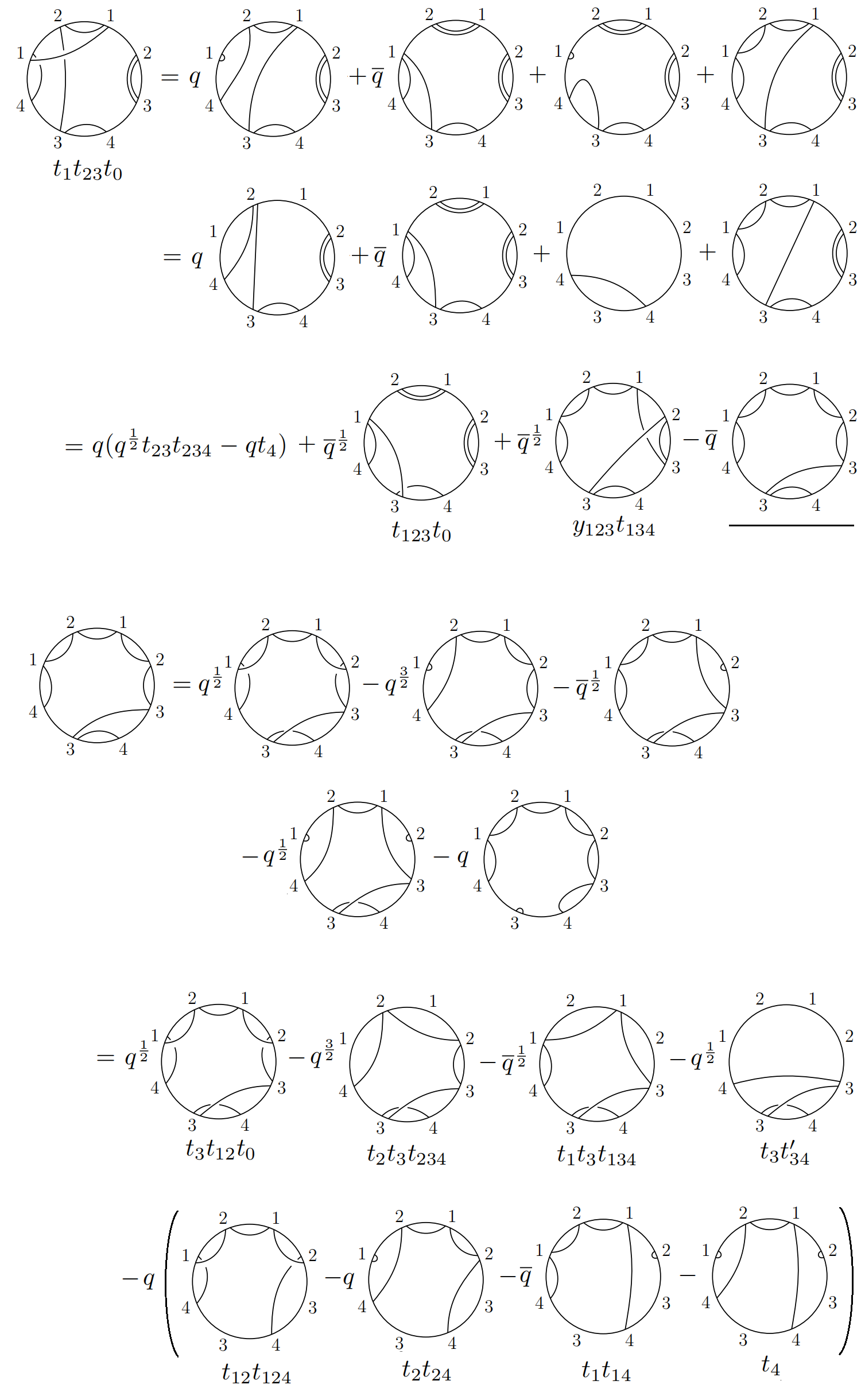}\\
  \caption{Deducing the formula for $t_{123}t_0$.}\label{fig:t123-times-t0}
\end{figure}

\begin{figure}[H]
  \centering
  \includegraphics[width=12.7cm]{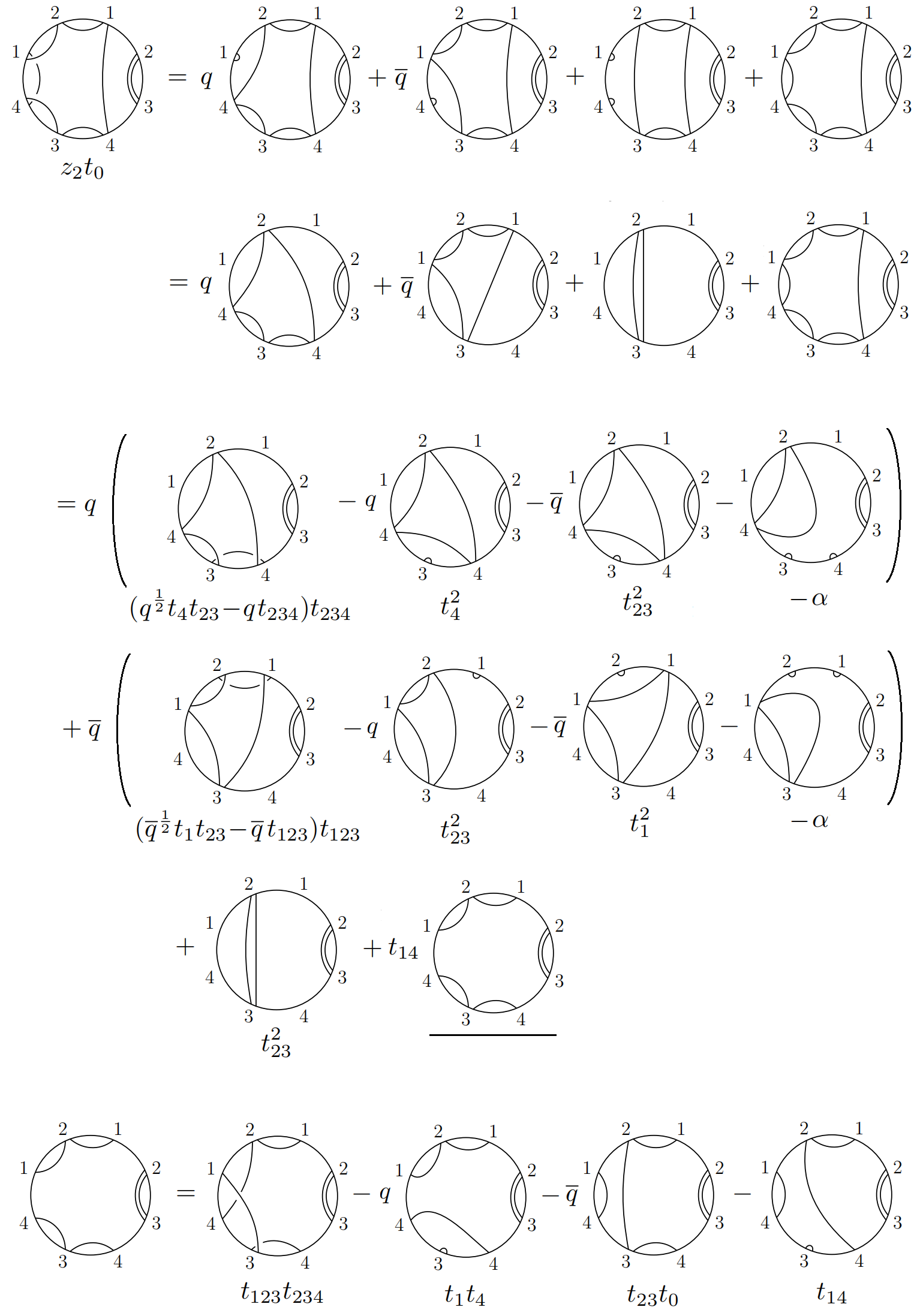}\\
  \caption{Deducing the formula for $t_0^2$.}\label{fig:t0-times-t0}
\end{figure}

\begin{figure}[H]
  \centering
  \includegraphics[width=11.3cm]{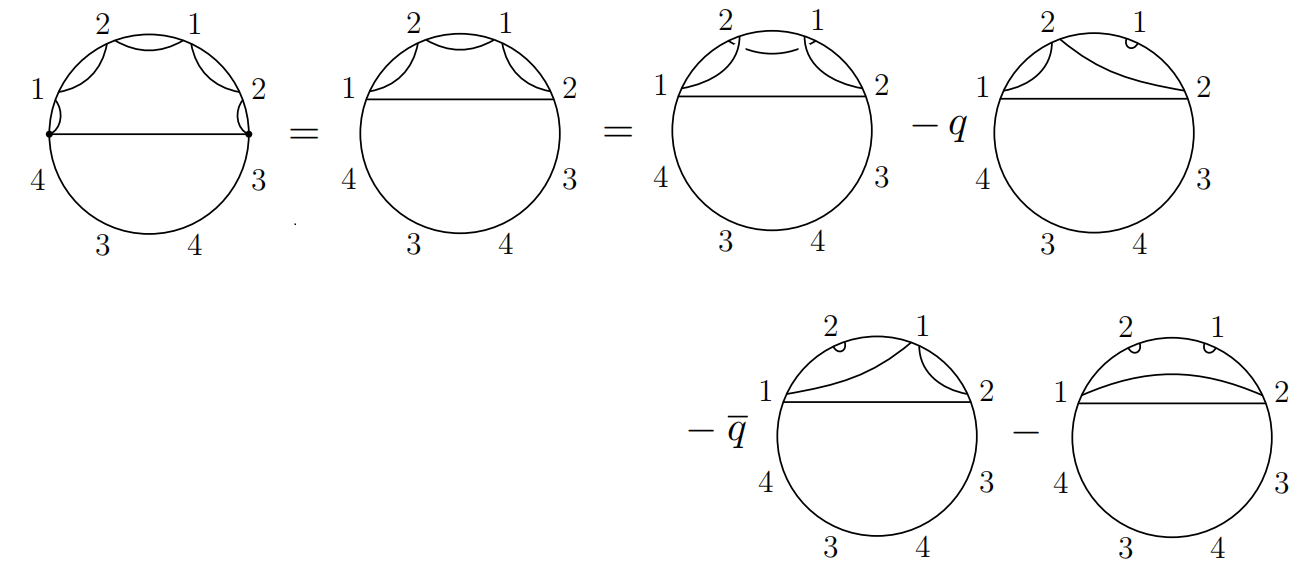}\\
  \caption{}\label{fig:relation-e}
\end{figure}

\begin{figure}[H]
  \centering
  \includegraphics[width=12.5cm]{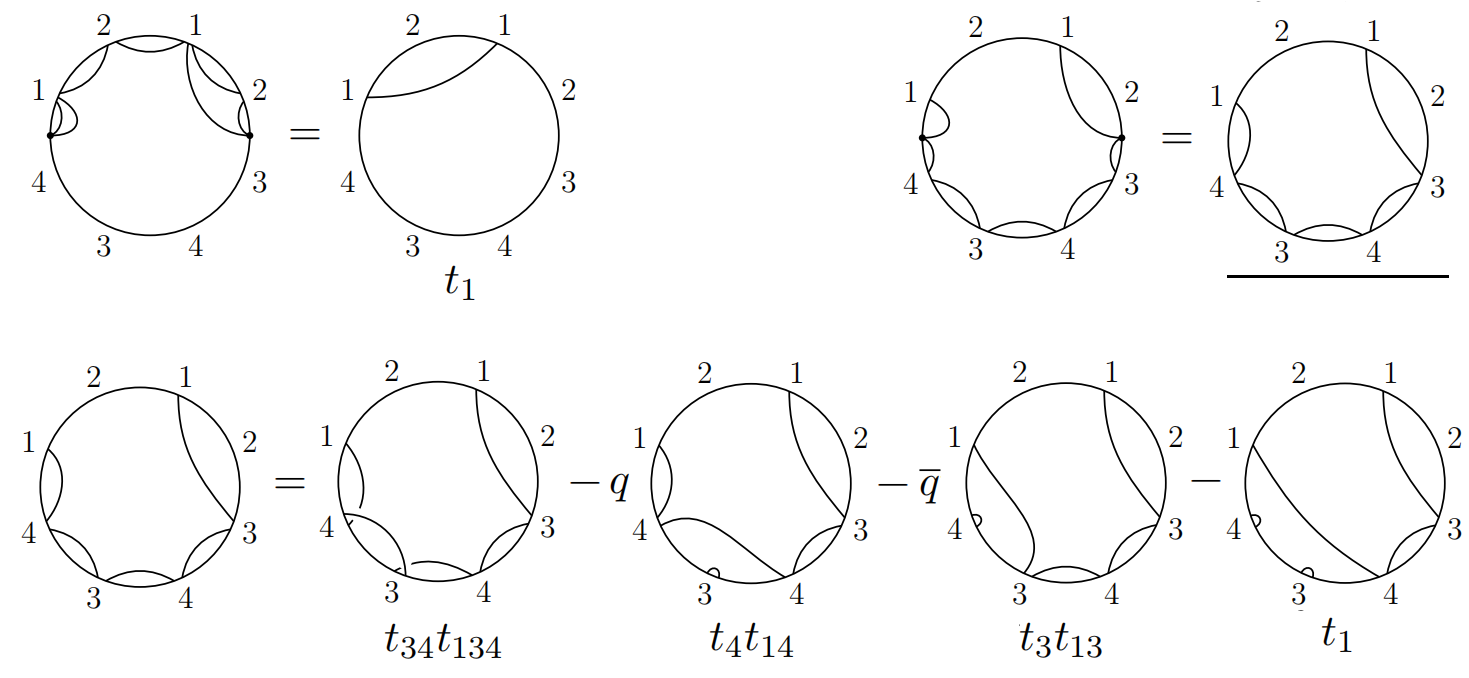}\\
  \caption{}\label{fig:relation-x1}
\end{figure}

\begin{figure}[H]
  \centering
  \includegraphics[width=12.2cm]{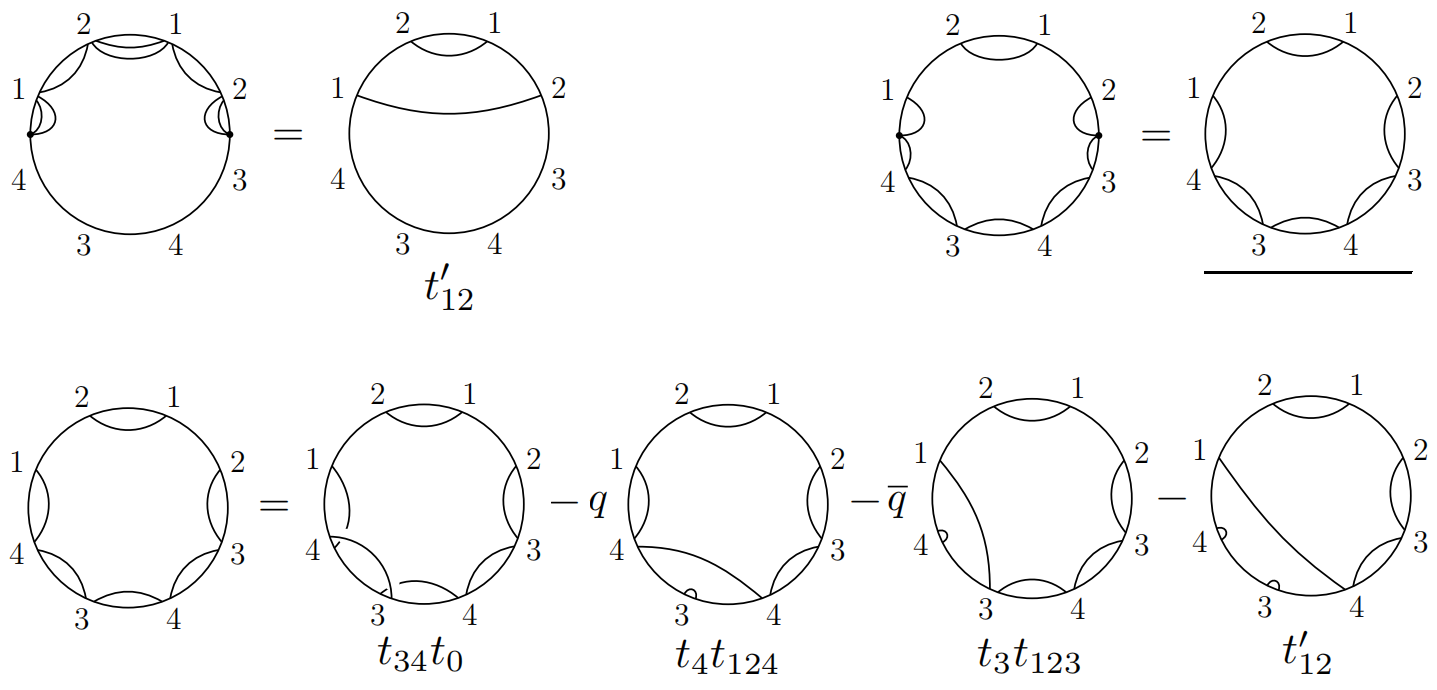}\\
  \caption{}\label{fig:relation-x12}
\end{figure}

\begin{figure}[H]
  \centering
  \includegraphics[width=11.5cm]{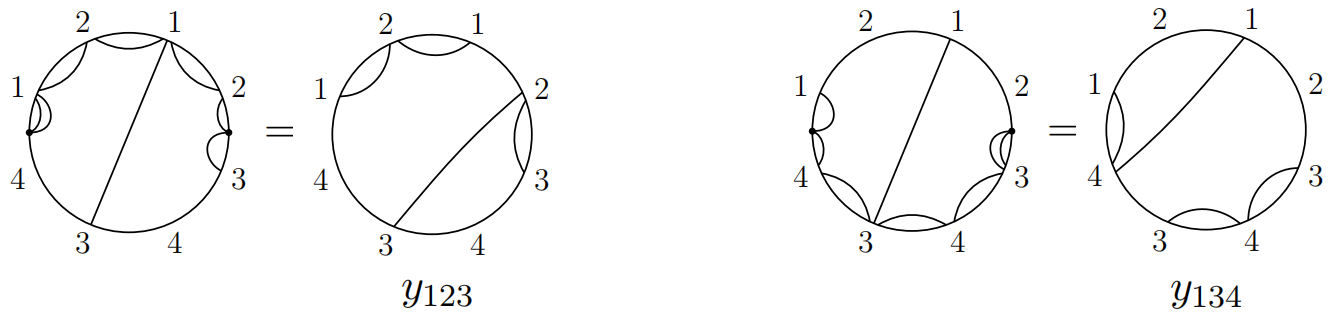}\\
  \caption{}\label{fig:relation-x13}
\end{figure}

\begin{figure}[H]
  \centering
  \includegraphics[width=12.7cm]{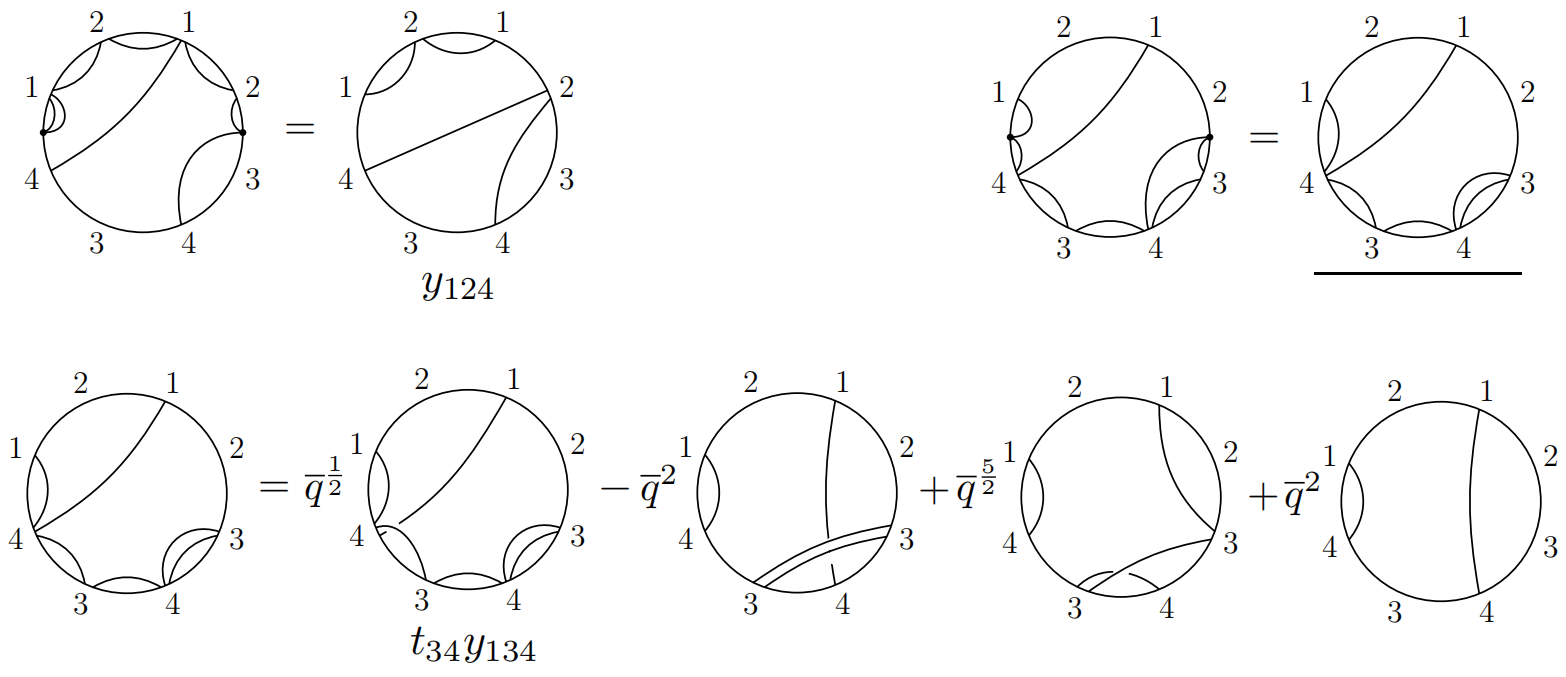}\\
  \caption{}\label{fig:relation-x14}
\end{figure}

\begin{figure}[H]
  \centering
  \includegraphics[width=12.7cm]{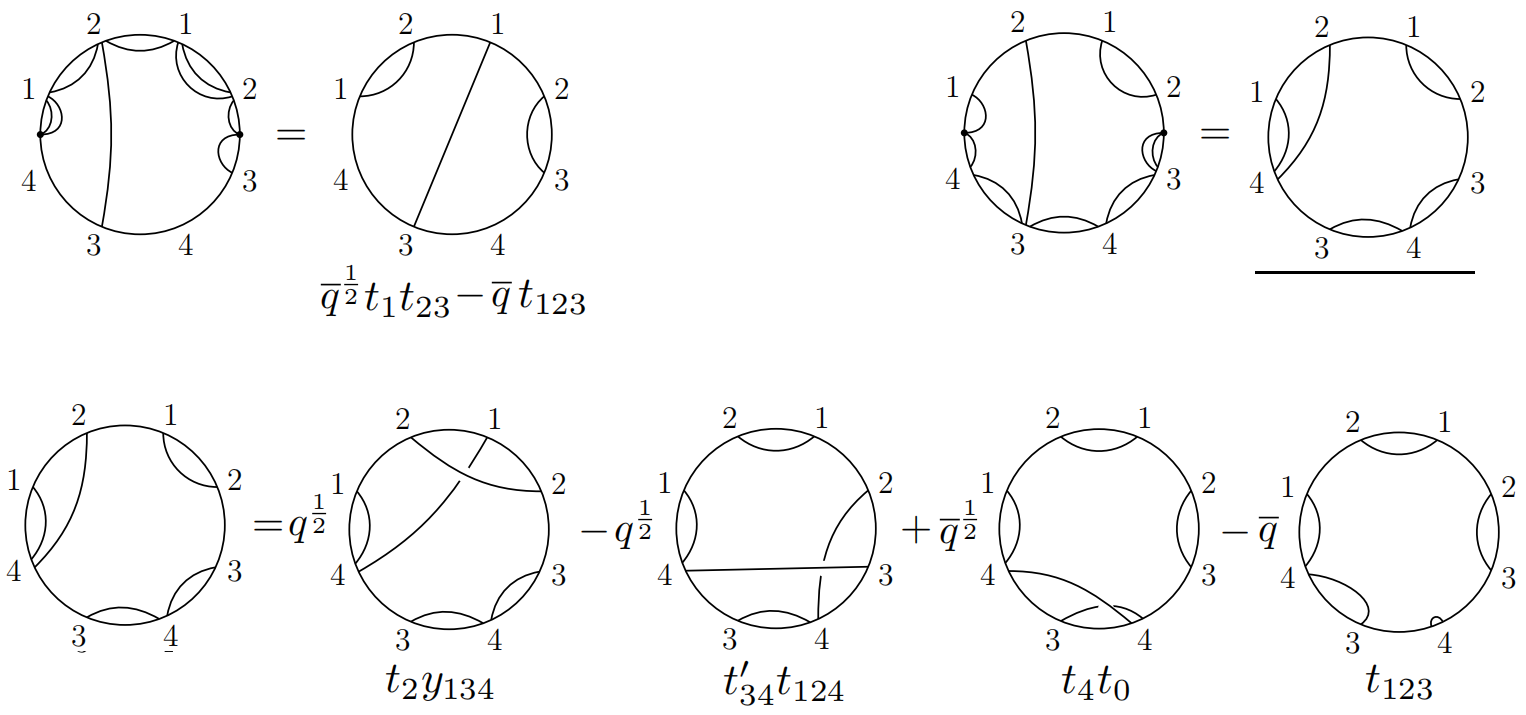}\\
  \caption{}\label{fig:relation-x123}
\end{figure}

\begin{figure}[H]
  \centering
  \includegraphics[width=12cm]{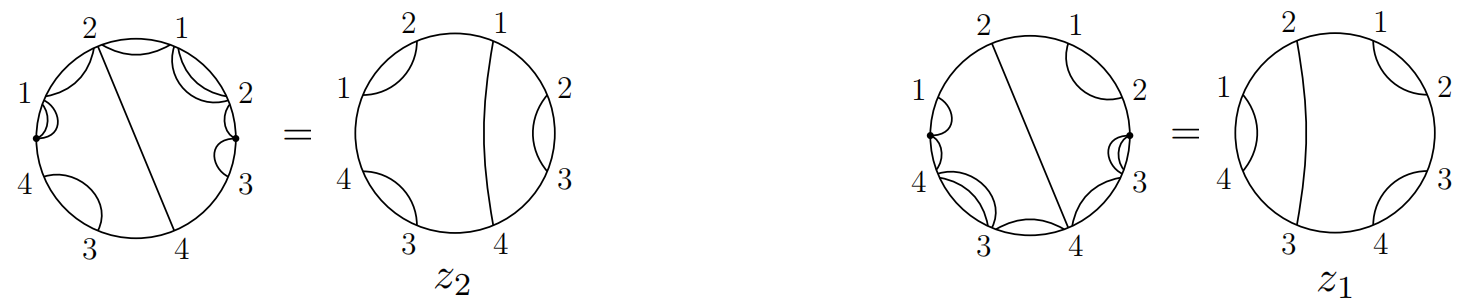}\\
  \caption{}\label{fig:relation-x1234}
\end{figure}

\bigskip

\noindent
Haimiao Chen (orcid: 0000-0001-8194-1264)\ \ \  \emph{chenhm@math.pku.edu.cn} \\
Department of Mathematics, Beijing Technology and Business University, \\
Liangxiang Higher Education Park, Fangshan District, Beijing, China.

\end{document}